\documentclass{amsart}

\usepackage{amsmath,amscd,amssymb,amsthm,amsfonts}
\usepackage{enumerate}
\usepackage[dvips]{graphicx}
\usepackage[all]{xy}

\newtheorem{theorem}{Theorem}[section]
\newtheorem{proposition}[theorem]{Proposition}
\newtheorem{corollary}[theorem]{Corollary}
\newtheorem{lemma}[theorem]{Lemma}

\theoremstyle{definition}
\newtheorem{definition}[theorem]{Definition}
\newtheorem{example}[theorem]{Example}

\theoremstyle{remark}
\newtheorem{remark}[theorem]{Remark}

\numberwithin{equation}{section}

\newcommand{\F}{\ensuremath{\mathcal{F}}}
\newcommand{\Flin}{\ensuremath{\mathcal{F}^{\ell}}}

\newcommand{\metric}{\ensuremath{ \mathrm{g} }}
\newcommand{\pr}{\ensuremath{\mathsf{p}}}

\newcommand{\ol}{\overline}

\usepackage{xcolor}
\usepackage{hyperref}

\title[Lie groupoids and semi-local models of SRF]{Lie groupoids and semi-local models of singular Riemannian foliations}

\author[Alexandrino]{Marcos M. Alexandrino}
\author[Inagaki]{Marcelo K. Inagaki}
\author[de Melo]{Mateus de Melo}
\author[Struchiner]{Ivan Struchiner}

\address{M. M. Alexandrino, M. K. Inagaki,  M. de Melo and  I. Struchiner  \hfill\break\indent 
Universidade de S\~{a}o Paulo, Instituto de Matem\'{a}tica e Estat\'{i}stica, \hfill\break\indent 
Rua do Mat\~{a}o 1010, 05508-090 S\~{a}o Paulo, Brazil.}

\email{(Alexandrino) marcosmalex@yahoo.de, m.alexandrino@usp.br}

\email{(Inagaki) kodi.inagaki@gmail.com}

\email{\textbf{(de Melo)  melomm@ime.usp.br}  }
\email{(Struchiner) ivanstru@ime.usp.br}



\thanks{The first author was supported by  grand $\#$ 2016/23746-6,  São Paulo Research Foundation (FAPESP).
The second author was supported  by CNPq (Conselho Nacional de Desenvolvimento Cient\'{\i}fico e Tecnol\'{o}gico). 
The third author was supported by grant $\#$ 2019/14777-3, São Paulo Research Foundation (FAPESP).
The fourth author was supported by grant   $\#$ 2015/22059-2, São Paulo Research Foundation (FAPESP)  and CNPq (307131/2016-5). 
In addition, this study was financed in part by the Coordena\c{c}\~{a}o de Aperfei\c{c}oamento de Pessoal de N\'{\i}vel Superior Brasil (CAPES)-Finance Code 001.
}

\begin{document}

%%%%%%%%%%%%%%%%%%%%%%%%%%%%%%%%%%%%%%%%%%%%%%%%%%
%%%%%%%%%%%%%%%%%%% Abstract %%%%%%%%%%%%%%%%%%%%%
%%%%%%%%%%%%%%%%%%%%%%%%%%%%%%%%%%%%%%%%%%%%%%%%%%
\begin{abstract}
We describe a local model for any Singular Riemannian Foliation in a neighbourhood of a closed saturated submanifold of a regular stratum. Moreover we construct a Lie groupoid which controls the transverse geometry of the linear approximation of the Singular Riemannian Foliation around these submanifolds.
We also discuss  the closure of this Lie groupoid and its Lie algebroid. 
\end{abstract}

\maketitle

%%%%%%%%%%%%%%%%%%%%%%%%%%%%%%%%%%%%%%%%%%%%%%%%%%
%%%%%%%%%%%%%%%%% Introduction %%%%%%%%%%%%%%%%%%%
%%%%%%%%%%%%%%%%%%%%%%%%%%%%%%%%%%%%%%%%%%%%%%%%%%
\section{Introduction}
In the theory of regular foliations a crucial role is played by the \emph{holonomy groupoid} of the 
foliation which gives a complete description of the geometry of the foliation transverse to its leaves. The holonomy groupoid of a foliation should be thought of as an atlas for the singular (and badly behaved) leaf space of the foliation.
%\textcolor{red}{In the theory of regular foliations a crucial role  is played by the \emph{holonomy groupoid} of the foliation $\F$
%whose orbits coincide with the leaves of the foliation $\F$.
%The restriction of the objects of the holonomy groupoid  to the disjoint union of
%slices (transverse  manifolds) can be thought of as an atlas for the
%singular (and badly behaved) leaf space of the foliation.}
For example, if the leaf space of a foliation $\mathcal{F}$ on a manifold $M$ is a manifold (or an orbifold), then the geometric structures it admits are in correspondence with geometric structures on on the normal bundle $\nu(\mathcal{F}) = TM/T\mathcal{F}$ which are invariant under the natural action of the holonomy groupoid of $\mathcal{F}$ (see for example \cite{Haefliger}). Moreover, when the holonomy groupoid of a foliation is well behaved (e.g., a proper/compact Lie groupoid) one obtains a simple explicit model for the foliation in a neighbourhood of a leaf known as the Reeb Stability Theorem (see \cite{Moerdijk-Mrcun,Reeb}, or \cite{Crainic-Struchiner}).

It is therefore natural to try to extend the construction of the holonomy groupoid to the case of singular foliations. The main attempt, so far, to obtain this generalisation has been made in \cite{Androulidakis-Skandalis} where a singular foliation is defined in terms of the module of vector fields which generates the foliation, see also \cite{Androulidakis-Zambon} and \cite{Garmendia-Zambon} . However, the groupoid constructed is rarely a Lie groupoid, and in this level of generality, it is possible that there does not exist a smooth groupoid describing the holonomy of the singular foliation.

In this paper we focus on a special case of singular foliations known as Riemannian foliations. We also consider a ``more geometric" approach to the definition of a singular foliation as a partition of the ambient manifold into leaves instead of fixing the module of vector fields chosen to generate it. Our main purpose is two-fold: on the one hand we describe a local model for any singular Riemannian foliation in a neighbourhood of a closed saturated submanifold of a  stratum, and on the other hand we obtain a holonomy groupoid for the linearization of a singular Riemannian foliation in a neighbourhood of such submanifolds. The holonomy groupoid that we obtain does not solve the problem of obtaining a Lie groupoid describing the transverse geometry of an arbitrary singular Riemannian foliation. This is still an open problem. However, our groupoid fits conceptually into this framework. Every Lie groupoid has a first order approximation around a saturated submanifold which is a transformation groupoid associated to a representation of the restriction of the original groupoid to the submanifold on the normal bundle of the submanifold \cite{Crainic-Mestre-Struchiner,Fernandes-Hoyo}. Even though the possible existence of a Lie groupoid describing the original singular Riemannian foliation is an open problem, what we obtain is a candidate for its first order approximation in a neighbourhood of any closed saturated submanifold of a regular stratum.   

We now explain in more details the results of this paper, but before we do so, we feel it is necessary to warn the reader that there are two distinct notions of holonomy that appear in this paper. The first one which already appeared above is the \emph{leafwise holonomy of a regular foliation}. The second one is the \emph{holonomy of a connection} obtained by parallel translations along paths for a fixed connection on a vector bundle (or a principal bundle). We hope that with this warning we will avoid unnecessary confusions and we will try to minimize this possibility by writing $\nabla$-holonomy for the second concept whenever it is not clear from the context.

\subsection*{Semi-Local Models for Singular Riemannian Foliations}

Given a Riemannian manifold $(M, \metric)$, a partition $\F = \{L\}$ of $M$ into complete connected submanifolds (the \emph{leaves of $\F$}) is called  a \emph{singular foliation} if every vector tangent to a leaf can be locally extended to a vector field everywhere tangent to the leaves (see \cite{Sussmann}). A singular foliation $\F$ is called \emph{Riemannian} (SRF for short) if each geodesic starting perpendicular to a leaf stays perpendicular to all leaves it meets.

A typical example of a SRF is the decomposition of a Riemannian manifold $M$ into the orbits of an isometric group action on $M$. Such a foliation is called \emph{homogeneous}. Another relevant example of a SRF is the \emph{$\nabla$-holonomy foliation} (presented below in Example \ref{holonomy-foliation}) which is related to other important types of foliations, like polar foliations \cite{Toeben} or Wilking's dual foliation to the Sharafutdinov projection \cite{Wilking}.

%%%%% example %%%%%
\begin{example}[$\nabla$-Holonomy foliation] \label{holonomy-foliation}
Let $B$ be a complete Riemannian manifold, $\mathbb{R}^{k} \to E \stackrel{\pi}{\rightarrow} B$ an Euclidean vector bundle (i.e. a vector bundle with a fiberwise metric) and $\nabla^{\tau}: \mathfrak{X} (B) \times \Gamma (E) \longrightarrow \Gamma (E)$ a linear connection which is compatible with the fiberwise metric of $E$.

Denote by $C^{\infty}([0, 1], M)$ the set of piecewise smooth curves in $M$ and by $\text{Hol}^{\tau}$ the $\nabla$-holonomy groupoid of $\nabla^{\tau}$ (i.e. the groupoid generated by all parallel transports along curves in $B$).
Then the partition $\F^{\tau} = \{\text{Hol}^{\tau}(v)\}_{v \in E}$ where
\begin{equation*}
\text{Hol}^{\tau}(v) := \left\{ \mathcal{P}_{\alpha} (v) \in E : \alpha \in C^{\infty}([0, 1], M), \; \pi(v) = \alpha(0) \right\}
\end{equation*}
is a SRF with respect to the \emph{Sasaki metric} on $E$ (i.e. if we denote by $\mathcal{T}$ the linear horizontal distribution on $E$ determined by $\nabla^\tau$, then the Sasaki metric is the metric which turns $T_v E = E_{\pi(v)} \oplus \mathcal{T} |_v$ an orthogonal decomposition, preserving the fiberwise metric of $E$ and the metric induced by the isomorphism $d \pi |_{\mathcal{T}}$ on $\mathcal{T}$).

We should stress three simple geometrical aspects of the $\nabla$-holonomy foliation. First considering the representation of $\text{Hol}^{\tau}$ on $E$ it is possible to see that the holonomy foliation is in fact given by the orbits of a groupoid, more precisely by the orbits of the transformation groupoid $\text{Hol}^{\tau} \ltimes E$ (see the definition in Section \ref{section-facts-groupoids}). Second, the intersection of the holonomy leaves with the fibers of $E$ are the orbits of holonomy groups. The last property is particularly special since there are infinitely many examples of non homogeneous SRFs on Euclidean spaces (see Radeschi \cite{Radeschi-clifford}). Third, $\mathcal{T}$ is tangent to $\F^{\tau}$ and $\mathcal{T} = T \F^{\tau}$ iff the connection $\nabla^{\tau}$ is flat (i.e., when $\F^{\tau}$ is regular foliation).
\end{example}

%%%%% Figura %%%%%
\begin{figure}[!htb]
\centering
\includegraphics[width=0.9\textwidth]{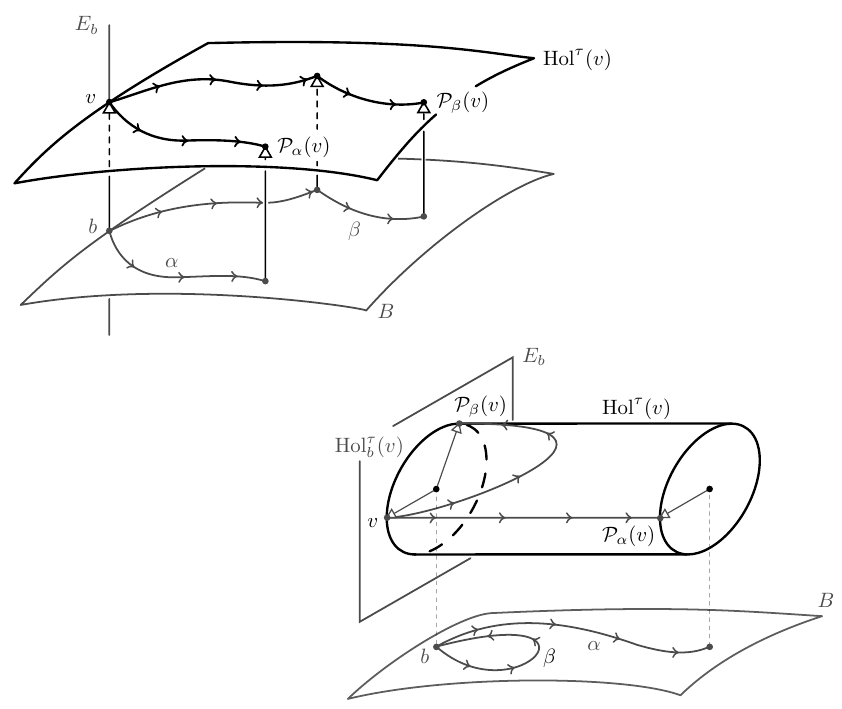}
\caption{Two possible illustrated schemes for holonomy foliation where $\mathcal{P}_\alpha$ and $\mathcal{P}_\beta$ denote the parallel transport along $\alpha$ and $\beta$, respectively.}
\label{holonomy}
\end{figure}

In \cite{Alexandrino-Radeschi-Molinosconjecture}, the first author and Radeschi proved the so called Molino's conjecture which states that \emph{given $(M, \F)$ a SRF on a complete manifold $M$, the partition $\ol{\F} = \{ \ol{L} \mid L\in \F\}$ of $M$ into the closures of the leaves of $\F$ is also a SRF}. In order to prove Molino's conjecture, they defined two foliations on an $\epsilon$-tubular neighbourhood $U$ around a closure of fixed leaf $L$. The first foliation was the so called \emph{linearized foliation} $\Flin$ of $\F$ in $U$. It is a subfoliation of $\F$ spanned by the first order approximations, around $\overline{L}$, of the vector fields tangent to $\F$. The second foliation denoted by $\widehat{\F}^\ell$ was then obtained from $\Flin$ by taking the ``local closure'' of the leaves of $\Flin$. Roughly speaking both foliations described the semi-local dynamical behavior of the foliation $\F$ (see Section \ref{preliminaries-foliation} for the definitions).

Let us now illustrate $\F,\,\Flin,\,\widehat{\F}^\ell$ in the prototypical examples (described bellow) which are convenient generalisations of example \ref{holonomy-foliation}.

%%%%% example %%%%%
\begin{example} \label{intermediate-holonomy-foliation}
Just like in the example \ref{holonomy-foliation} let $B$ be a complete Riemannian manifold, $\mathbb{R}^{k}\to E \stackrel{\pi}{\rightarrow} B$ an Euclidean vector bundle and $\nabla^{\tau}$ a connection which is compatible with the fiberwise metric of $E$.

Additionally to the previous data in Example \ref{holonomy-foliation}, consider
\begin{itemize}
\item $\F^E = \{L^{E}_{v}\}_{v \in E}$ a singular foliation on $E$ such that each fiber $E_{b}$ is saturated and $\F_{b}:= \F^{E}|_{E_{b}}$ is an \emph{infinitesimal foliation} (i.e. $\F_b$ is a SRF on a vector space $E_{b}$ with $\{0_b\}$ as a leaf). Assume that $\F^E$ is $\text{Hol}^\tau$-invariant (i.e. the parallel transport sends leaves into leaves).
\end{itemize}
Denote by $K^{0}_{b}$ the maximal connected Lie subgroup of isometries of $E_{b}$ that fixes each leaf of $\F_{b}$
and $\overline{K^{0}_{b}}$ the closure of $K_{b}^{0}$ for each $b \in B$ (see the discussion in Example \ref{example-linearization}).
 Then the following three partitions of $E$ are in fact smooth singular foliations:
\begin{enumerate}
\item[(a)] $\F = \{ \text{Hol}^\tau (L^{E}_{v}) \}_{v \in E}$;

\item[(b)] $\Flin = \{ \text{Hol}^\tau (K^{0}_{\pi(v)} (v)) \}_{v \in E}$;

\item[(c)] $\widehat{\F}^\ell = \{ \text{Hol}^\tau (\overline{K_{\pi(v)}^{0}}(v)) \}_{v \in E}$.
\end{enumerate}

For the Sasaki metric on $E$ (described in Example \ref{holonomy-foliation}), the singular foliation $\F$ turns out to be a SRF, $\Flin$ becomes its linearized foliation, and $\widehat{\F}^\ell$ becomes the local closure of $\Flin$.

\end{example}

Example \ref{intermediate-holonomy-foliation} is in fact the semi-local model of a SRF in a $\epsilon$-tubular neighborhood of a closed leaf $L = B$ (see Theorem \ref{theorem-semi-local-model}). In order to obtain a semi-local model around more general closed saturated submanifolds of a  stratum, one must also take into account the restriction of the original foliation to the submanifold. This leads us to the following generalisation of the previous example.

%%%%% example %%%%%
\begin{example}[Generalized holonomy foliation] \label{generalized-holonomy-foliation}
As in the previous example, let $B$ be a complete Riemannian manifold, and $\mathbb{R}^{k}\to E \stackrel{\pi}{\rightarrow} B$ an Euclidean vector bundle endowed with a singular foliation $\F^E = \{L^{E}_{v}\}_{v \in E}$ such that each fiber $E_{b}$ is saturated and $\F_{b}:= \F^{E}|_{E_{b}}$ is an infinitesimal foliation . This time we consider also a (regular) Riemannian foliation $\F_B$ on $B$ and we take $\nabla^\tau$ to be an $\F_B$-partial connection on $E$ which is compatible with the fiberwise metric on $E$, and such that $\F^E$ is invariant under parallel translation with respect to $\nabla^\tau$.

If we denote by $K^{0}_{b}$ the maximal connected group of isometries of $E_{b}$ that fixes each leaf of $\F_{b}$ and by $\overline{K^{0}_{b}}$ the closure of $K_{b}^{0}$ for each $b \in B$, then we obtain three foliations on $E$:

\begin{enumerate}
\item[(a)] $\F = \{ \text{Hol}^{\tau} (L^{E}_{v}) \}_{v \in E}$ where $L^{E}_{v}\in \F^{E}$;

\item[(b)] $\Flin = \{ \text{Hol}^{\tau} (K^{0}_{\pi(v)} (v)) \}_{v \in E}$;

\item[(c)] $\widehat{\F}^\ell = \{ \text{Hol}^{\tau} (\overline{K_{\pi(v)}^{0}}(v)) \}_{v \in E}$,
\end{enumerate}
where $Hol^\tau$ is generated by parallel translations along paths in the leaves of $\F_B$, with respect to $\nabla^\tau$.

It follows that there is a Sasaki metric on $E$ such that the foliation $\F$ turns out to be a SRF, $\Flin$ becomes its linearized foliation, and $\widehat{\F}^\ell$ becomes the local closure of $\Flin$. 
\end{example}

Our first theorem states that the example above is in fact a local model for any SRF in a small tubular neighbourhood of any closed saturated submanifold of a  stratum.

%%%%% theorem %%%%%
\begin{theorem} \label{theorem-semi-local-model}
Let $\F$ be a singular Riemannian foliation on a complete manifold $(M,g)$ and $B$ be a closed saturated submanifold contained in a stratum of the foliation. Then there exists a saturated $\epsilon$-tubular neighborhood $U$ of $B$ in $M$ such that the foliations $\F,\,\Flin,\,\widehat{\F}^\ell$ restricted to $U$ are foliated diffeomorphic to the foliations described in the Example \ref{generalized-holonomy-foliation} where the Euclidean vector bundle $E$ is the normal bundle of $B$.
\end{theorem}

The main ingredients of the proof of item were already presented in \cite{Alexandrino-Radeschi-Molinosconjecture}. Here we put  these ingredients together with help of  Proposition \ref{lemma-linearized-vector-linearfoliation} stressing its semi-local description (see also the discussion in Mendes and Radeschi \cite{Mendes-Radeschi} for the case where $B$ is a closed leaf).

%%%%%%%%%%%%%%%%%%%%%%%%%%%%%%%%%%%%%%%%%%%%%%%%%%%%
%%%%%%%%%%%%%%%%%%%%%%%%%%%%%%%%%%%%%%%%%%%%%%%%%%%%
%%%%%%%%%%%%%%%%%%%%%%%%%%%%%%%%%%%%%%%%%%%%%%%%%%%%
%%%%%%%%%%%%%%%%%%%%%%%%%%%%%%%%%%%%%%%%%%%%%%%%%%%%

\subsection*{Lie Groupoids and Singular Riemannian Foliations}
Lie groupoids are structures which generalize smooth manifolds, Lie groups and Lie group actions (see \cite{Moerdijk-Mrcun} or Section \ref{section-facts-groupoids} below). Every Lie Groupoid determines a singular foliation on its base manifold by taking  (the connected components of) its orbits. It is not yet known how to characterise the singular foliations which arise as orbits of a Lie groupoid. Moreover, even when a singular foliation arises in this fashion, there does not exist in general a canonical Lie groupoid which describes it. In contrast, to any regular foliation one can associate two canonical Lie groupoids which describe it, namely, the monodromy and the holonomy groupoids of the foliation (see Section \ref{section-facts-groupoids}). For a given regular foliation $\mathcal{F}$ on $M$, the monodromy groupoid $\Pi_1(\mathcal{F})$ is the unique source simply connected Lie groupoid which integrates the Lie algebroid $T\mathcal{F}$, while the holonomy groupoid $\mathrm{Hol}(\mathcal{F})$ is the terminal object in the category of source connected Lie groupoids which has $\mathcal{F}$ as its orbit foliation (see \cite{Crainic-Moerdijk-foliation}). 

The orbit foliation associated to a Lie groupoid is not in general a SRF. However, if the groupoid is proper, or more generally if it admits a compatible Riemannian metric (a Riemannian Groupoid), then the orbit foliation is in fact a SRF (see \cite{Fernandes-Hoyo,Hessel-Pflaum-Tang}). An important feature of Riemannian Groupoids is that they can be linearized around saturated submanifolds \cite{Fernandes-Hoyo}. In a nutshell, this means that the restriction of a Lie groupoid $\mathcal{G}\rightrightarrows M$ to a small tubular neighbourhood of a saturated submanifold $B$ of $M$ is locally isomorphic to a transformation groupoid associated to a representation (i.e., a linear action) of the restriction of $\mathcal{G}$ to $B$ on the normal bundle $E$ of $B$ in $M$.

It is therefore natural to assume that if there exists a holonomy groupoid associated to a SRF, then it should be linearizable. In this paper we construct a canonical linear groupoid associated to the linearized foliation $\mathcal{F^\ell}$ in a tubular neighbourhood of a closed saturated  submanifold contained in a regular stratum. This is the content of the following theorem.

\begin{theorem} \label{theorem-linear-holonomy-groupoid}
Let $\F$ be a singular Riemannian foliation on a complete manifold $(M,g)$ and $B$ be a closed saturated submanifold contained in a stratum of the foliation. Then there exists a saturated $\epsilon$-tubular neighborhood $U$ of $B$ in $M$ such that the leaves of the foliations $\Flin,\,\widehat{\F}^\ell$ restricted to $U$ are orbits of a canonical  transformation  groupoid associated to a representation of a regular groupoid $\mathcal{G} \rightrightarrows B$ on the normal bundle $E$ of $B$ in $M$.
\end{theorem}

We argue that the transformation groupoid whose orbits are the leaves of $\F^{\ell}$ obtained in the theorem above should be thought of as the linearization of the holonomy groupoid of the SRF $\F$ around $B$ (even though we do not know if such a holonomy groupoid exists). For this reason, we call this groupoid the  \emph{Linear Holonomy Groupoid} of $\F$ at $U$.

%%%%%%%%%%%%%%%%%%%%%%%%%%%%%%%%%%%%%%%%%%%%%%%%%%%%%%%%%
%%%%%%%%%%%%%%%%%%%%%%%%%%%%%%%%%%%%%%%%%%%%%%%%%%%%%%%%%
%%%%%%%%%%%%%%%%%%%%%%%%%%%%%%%%%%%%%%%%%%%%%%%%%%%%%%%%%%

The construction of the Linear Holonomy Groupoid of $\F$ at $U$ relies on the fact that we can ``lift" the foliation $\Flin$ to an $O(n)$-invariant regular foliation on the orthogonal frame bundle of the normal bundle $E$ of $B$ in $M$. The (usual) holonomy groupoid  of this (regular) lifted foliation comes with a free action of $O (n)$ by automorphisms, and the quotient groupoid comes with a canonical representation on E.

It is possible to check that $\overline{\F^{\ell}}=(\overline{\F})^\ell$
and hence to conclude that
 the leaves of SRF  $\overline{\F^{\ell}}$ 
are also orbits of a Lie groupoid. It is then a natural to ask what is
the relation of the 
Linear Holonomy Groupoid  $\mathcal{G}^{\ell}$ that describe $\F^{\ell}$ 
and this ``bigger'' groupoid. 
 
\begin{theorem} 
\label{prop:subgrupoids}
Let $\F$ be a singular Riemannian foliation on a complete manifold $(M,g)$ and $B = \ol{L}$. 
Then there exists a proper Lie groupoid $\overline{\mathcal{G}^{\ell}}$ over  a saturated $\epsilon$-tubular neighborhood $U$ of $B$ whose orbits are 
 the leaves  of $\overline{\F^{\ell}}$. In addition $\mathcal{G}^{\ell}$ is a dense Lie subgroupoid of $\overline{\mathcal{G}^{\ell}}$.  
\end{theorem}

%%%
%\footnote{IVAN esta era a observacao correta, aparentemente tinha um remark errado ao fim da secao 5 que eu tornei invisivel. Se quiser discutimos sobre isto}
%%%%%
\begin{remark}
Let $\F$ be a Riemannian foliation on $B$. Then the fact that $B=\overline{L_{q}}$ implies that for all $b\in B$ we have that 
$B=\overline{L_b}$, i.e., that the foliation is dense. 
\end{remark}

\subsection*{ $\F$-partial connection  and Lie Algebroid} 
%%
%%%%%
In the particular case when the regular foliation $\F_B$ is \emph{a dense foliation} (each leaf  $L\in \F_B$ is dense) and $U$ is a saturated $\epsilon$-tubular neighborhood around $B$ we have a second subfoliation $\F^{\tau}$ so that $$\F^{\tau} \subset \F^{\ell} \subset \F_U.$$ Roughly speaking, the subfoliation $\F^{\tau}$ could be thought of as the foliation produced just by taking parallel transports of normal vectors with respect to some $\F$-partial connection. In other words, if we consider Example \ref{generalized-holonomy-foliation} applied to the particular case where $\F_B$ is a dense foliation, $E$ is the normal bundle of $B$, and $\F^E$ is the foliation given by the points of $E$, then the partition $\F^{\tau} = \{L^{\tau}_{v}\}_{v \in E}$ which has leaves $L^{\tau}_{v} := \text{Hol}^\tau (v)$ for $v \in E$ is a singular (smooth) subfoliation of $\F$. The next result also assure that the leaves of $\F^{\tau}$ are orbits of a Lie groupoid.
More generally, we can consider a foliation
$ \F^{\tau} $ induced by a partial connection
$ \nabla^{\tau} $, just starting with a dense 
foliation $ \F_B $, without assuming that it is a  Riemannian
foliation. 

%%%%% theorem %%%%%
\begin{theorem} \label{foliated-Ambrose-and-Singer}
Let $\mathbb{R}^{n}\to E\to B$ be an Euclidean vector bundle. Assume that there exists a dense foliation $\F_{B}$ on the basis $B$
(each leaf  $L\in \F_B$ is dense) 
 and a partial linear connection $\nabla^{\tau}: \mathfrak{X}(\F_{B})\times \Gamma(E)\to \Gamma(E)$ compatible with the metric of $E$. Consider the singular partition $\F^{\tau} = \{L^{\tau}_{v}\}_{v \in E}$, with leaves $L^{\tau}_{v} = \text{Hol}^{\,\tau} (v)$ where $\text{Hol}^{\,\tau}$ is the holonomy groupoid of $\nabla^{\tau}$. Then the leaves of the  singular foliation $\F^{\tau}$ are orbits of a transformation groupoid associated to a representation of a Lie groupoid over $B$ on $E$.
\end{theorem}

The Lie groupoid   over $B$ 
of the previous theorem is constructed only in terms of the $\F$-partial connection $\F$. 
It is a Lie subgroupoid of the Gauge Groupoid associated to the frame bundle of $E$. In this sense, the theorem above may be thought of as a generalisation of the Ambrose-Singer reduction theorem to the context of foliated connections. 
We will explore this point of view even further in Section \ref{section-rotate-translate-groupoid}
%\ref{section-facts-algebroids} 
when we describe the infinitesimal objects associated to the groupoids of this paper, i.e., the \emph{Lie algebroids} (see Section \ref{section-facts-algebroids}). This will give also an alternative way of obtaining the Lie groupoids via integration of their algebroids  
(see Definition \ref{definition-algebroid}
and  Propositions  \ref{proposition-algebroid-puro} and  \ref{proposition-algebroid-SRF}).

\subsection*{Organization of the Paper}
%%%%

%%
This paper is organized as follows. In Section \ref{preliminaries-foliation} we review preliminary facts needed in the rest of the paper. 
In particular we review the main ingredients of the theory of SRF that will allow us  to prove 
  Theorem \ref{theorem-semi-local-model} (see Section \ref{section-semi-local-modelsofSRF}).
 In Section \ref{section-lie-groupoid-structure} we discuss the Lie groupoid structure needed to prove 
Theorem \ref{theorem-linear-holonomy-groupoid}
 and Theorem \ref{foliated-Ambrose-and-Singer}. In these proofs, the  frame bundles associated to the foliation are used in a natural way. The proof indicates that orthogonal frame bundles can play an important role in the study of SRF as they have been playing in the study of (regular) Riemannian foliations (see Molino's book \cite{Molino}). 
We give in Section \ref{section-subgroupoid-closure} 
%a proof that the closure of a regular Riemannian manifold is presented by a Lie groupoid, 
%and use it to prove 
the proof of 
Theorem \ref{prop:subgrupoids}. 
%
%\footnote{apaguei a mensao na introducao da tal proposicao que aparentemente ja esta ersolvida em [21] }
%
 Finally in Section \ref{section-rotate-translate-groupoid} we remark that  the existence of the Lie groupoid structure in  Example \ref{generalized-holonomy-foliation} does not require the hypothesis that $\F_{B}$ is a Riemannian foliation on $B$ and stress its Lie algebroid structure   (Propositions  \ref{proposition-algebroid-puro} and  \ref{proposition-algebroid-SRF}).
By moving from the concrete case of SRF to abstract considerations in  Section \ref{section-rotate-translate-groupoid}, we hope to provide the proper motivation for readers who are not experts in the theory of Lie groupoids.

%%%%%%%%%%%%%%%%%%%%%%%%%%%%%%%%%%%%%%%%%%%%%%%%%%%%%%%%%%
%%%%%%%%%%%%%%%%%%%% ACKNOWLEDGEMENTS %%%%%%%%%%%%%%%%%%%%
%%%%%%%%%%%%%%%%%%%%%%%%%%%%%%%%%%%%%%%%%%%%%%%%%%%%%%%%%%
\section*{Acknowledgements}
The authors thank Fernando M. Escobosa for useful suggestions.

%%%%%%%%%%%%%%%%%%%%%%%%%%%%%%%%%%%%%%%%%%%%%%%%%%%%%%%
%%%%%%%%%%%%%%%%%%%% PRELIMINARIES %%%%%%%%%%%%%%%%%%%%
%%%%%%%%%%%%%%%%%%%%%%%%%%%%%%%%%%%%%%%%%%%%%%%%%%%%%%%
\section{Preliminaries} \label{preliminaries-foliation}

%%%%%%%%%%%%%%% A few facts about SRF %%%%%%%%%%%%%%%
\subsection{A few facts about SRF} \label{section-facts-SRF}

In this section we briefly review a few facts on a SRF $(M, \F)$ extracted from \cite{Alexandrino-Radeschi-Molinosconjecture} (see also \cite{Mendes-Radeschi} and \cite{Molino}).

%%%%%%%%%%%%%%% Linearization of vector fields %%%%%%%%%%%%%%%
\subsubsection{Linearization of vector fields}

Let $B$ be a closed saturated submanifold contained in a stratum, e.g., the closure of a leaf ($B = \overline{L}$) or the minimal stratum of $\F$, and let $U\subset M$ be an $\epsilon$-tubular neighbourhood of $B$ with metric projection $\pr: U \to B$. For each smooth vector field $\vec{V}$ in $U$, we can associate a smooth vector field $\vec{V}^\ell$, called the \emph{linearization of $\vec{V}$ with respect to $B$}, as follows:
\[
\vec{V}^\ell = \lim_{\lambda \to 0} \vec{V}_{\lambda}
\]
where $\vec{V}_{\lambda} |_{q} := (h_{\lambda}^{-1})_*(\vec{V} |_{h_{\lambda}(q)})$ and $h_\lambda: U \to U$ denotes the homothetic transformation around $B$, i.e., the map given by $h_{\lambda} (\exp(v)) := \exp(\lambda v)$ for each $v \in \nu^\epsilon (B)$ and $\lambda \in (0, 1]$. Since the plaques of $\F$ are invariant under homothetic transformation, one can conclude that if $\vec{X}$ is a vector field tangent to $\F$, then $\vec{X}^{\ell}$ is still tangent to the foliation $\F$.

%%%%% example %%%%%
\begin{example}[Infinitesimal foliation] \label{example-linearization}
Consider a SRF $\F$ on $\mathbb{R}^n$, where $B=\{0\}$ is a closed leaf. Given a smooth vector field $\vec{X}$  tangent to the leaves, the associated linearized vector field is given by
\[
\vec{X}^{\ell}_v = \lim_{\lambda\to 0}(h_{\lambda}^{-1})_* \vec{X}_{ h_{\lambda}(v)} = \lim_{\lambda\to 0}\frac{1}{\lambda} \vec{X}_{\lambda v} = \left(\nabla_{v} \vec{X}\right)_{0}
\]
 Note that the linear vector field $\vec{X}^{\ell}_{(\cdot)}=(\nabla_{(\cdot)} \vec{X})_{0}$ is determined by a matrix  which is skew-symmetric and hence $\vec{X}^{\ell}$ is a Killing vector field. In fact
since $\vec{X}^{\ell}$ is tangent to the leaves, it is tangent to the distance spheres around $0$, and therefore $$0 = \langle \vec{X}^{\ell}_v,v \rangle = \left\langle \left(\nabla_{v} \vec{X}\right)_0, v \right\rangle,$$ for all unitary $v$.
%%%%
We can define $\mathfrak{k}$  
as the maximal Lie subalgebra of $\mathfrak{o}(n)$ 
that induces these Killing vector fields fields and $K^{0}$  the connected subgroup  of $SO(n)$ that has Lie algebra $\mathfrak{k}$. It is not difficult to check 
that $K^0$ is the maximal connected Lie subgroup of $SO(n)$ that fixes each leaf of $\F$. In addition, when the leaves of $\F$ are compact,
$K^{0}$ is compact. 

\end{example}

%%%%
%%%%%%%%%%%%%%% The linearized foliation Fl%%%%%%%%%%%%%%%
%%%
\subsubsection{The linearized foliation $\Flin$} \label{subsection-linearized foliation}

Let $\mathsf{D}$ be the pseudogroup of local diffeomorphisms of $U$, generated by the flows of linearized vector fields tangent to $\F$. Then the partition of $U$ into the orbits of diffeomorphisms in $\mathsf{D}$ is called the \emph{linearized foliation of $\F$ (with respect to $B$)} and is denoted by $(U,\Flin).$

Since the linearization of vector fields tangent to $\F$ are still vector fields tangent to $\F$, one concludes that $\F^{\ell}$ is a subfoliation of $\F$. In other words, the leaves of $\F^{\ell}$ are contained in the leaves of $\F$ and $\F^{\ell}|_{B}$ coincides with $\F|_{B}$.

We now recall that $\F^{\ell}$ is the \emph{maximal infinitesimal homogeneous subfoliation of $\F$}. In fact, given a point $b \in B$, denote $U_b := \pr^{-1}(b) \subset U$ and let $\F_b$ (resp. $\F^{\ell}_{b}$) denote the partition of $U_b$ into the connected components of $\F \cap U_b$ (resp. $\F^{\ell}\cap U_b$). It was proved in \cite[Propositions 6.5]{Molino} that $\F_b$ turns out to be a SRF on the Euclidean space $(U_b, \metric_b)$ called the \emph{(reduced) infinitesimal foliation at $b$}, once we identify $U_b$ (via exponential map) with an open set of $\nu_b (B)$ with the flat metric $\metric_b$. In addition the  foliation $(U_b, \F^{\ell}_{b})$ is homogeneous, more precisely the leaves of $\F^{\ell}_{b}$ are orbits of $K^{0}_{b}$, where $K^{0}_{b}$ denotes the maximal connected Lie group of isometries of $\nu_b (B)$ that fixes each leaf of $\F_b$ (cf. Example \ref{example-linearization}).

More generally, given a vector field $\vec{X}$ tangent to the leaves, the flow $\varphi_t$ of the linearized vector field $\vec{X}^{\ell}$ induces an isometry between $(U_b, \metric_b)$ and $(U_{\varphi_t (b)}, \metric_{\varphi_t (b)})$.

%%%
%%%%%%%%%%%%%%% The distributions K, T an N %%%%%%%%%%%%%%%
%%
\subsubsection{The distributions $\mathcal{K}$, $\mathcal{T}$ and $\mathcal{N}$} \label{subsection-3-distribution}

Let us now review the definition of the 3 important distributions necessary to understand the semi-local model of $\F$.
\begin{itemize}
\item $\mathcal{K}=\ker\pr_*$ where $\pr: U \to B$.

\item There exists a distribution $\widehat{\mathcal{T}} \subset T U$ tangent to the leaves of $\F$, such that $\widehat{\mathcal{T}} |_B = T \F |_{B}$ (for a construction of such distribution see \cite[Proposition 3.1]{Alexandrino-desingularization}). The \emph{distribution $\mathcal{T}$ is the linearization of $\widehat{\mathcal{T}}$} with respect to $B$. \newline We point out that $\mathcal{T}$ is tangent to $\F$ since it is a linearization of a distribution tangent to the leaves and the rank of $\widehat{\mathcal{T}}$ and $\mathcal{T}$ are equal to $\dim \F|_B$.

\item Let $b = \pr(q)$ and $\widehat{\mathcal{N}}_q$ be the subspace of $T_q S_b$ which is $g_b$-orthogonal to $\mathcal{K}_q$ where $S_b$ is the slice of $\F$ at $b$. The \emph{distribution $\mathcal{N}$ is the linearization of $\widehat{\mathcal{N}}$ with respect to $B$}.
\end{itemize}
Note that all three distributions are homothetic invariant and $TU = \mathcal{K} \oplus H$ with $H := \mathcal{T}\oplus\mathcal{N}$ (see figure \ref{TNK}). We quickly conclude that $\mathcal{K}$, $\mathcal{T}$ and $\mathcal{N}$ can be identified with homothetic invariant distributions on the normal bundle $\nu(B)$ and particularly $H$ is identified with a horizontal distribution in $\nu(B)$.

%%%%% figura %%%%%
%\begin{figure}[!htb]
%\centering
%\resizebox{0.6\textwidth}{!}{\input{DESENHO-TKN.pdf_tex}}
%%%%%%%%%\input{TKN.pdf_tex}
%\caption{Illustrated scheme of the distributions $\mathcal{K}$, $\mathcal{T}$ and $\mathcal{N}$.}
%\label{TNK}
%\end{figure}

%%%%% Figura %%%%%
\begin{figure}[!htb]
\centering
\includegraphics[width=0.6\textwidth]{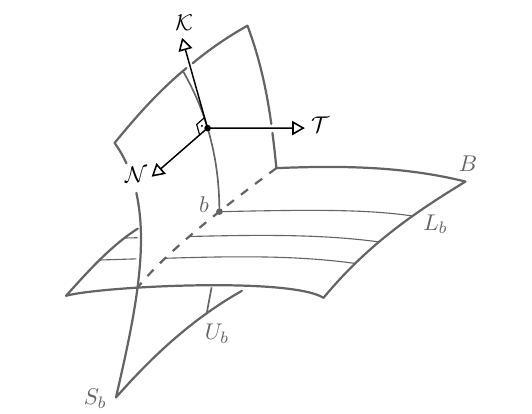}
\caption{Illustrated scheme of the distributions $\mathcal{K}$, $\mathcal{T}$ and $\mathcal{N}$.}
\label{TNK}
\end{figure}

The next two results are consequences of a discussion presented in Section 5 of \cite{Alexandrino-Radeschi-Molinosconjecture}.
%%% (see e.g, Proposition 5.1 ) %%%

%%%%% proposition %%%%%
\begin{proposition} \label{proposition-connections-tau-affine}
The homothetic invariant distribution $\mathcal{T} \oplus \mathcal{N}$ induces a linear connection $\nabla: \mathfrak{X}(B)\times \Gamma(\nu(B))\to \Gamma(\nu(B))$ that when restricted to $\F_{B}$ induces a partial linear connection $\nabla^{\tau}: \mathfrak{X}(\F_{B})\times \Gamma(\nu(B))\to \Gamma(\nu(B))$ compatible with the metric of $\nu(B)$.
\end{proposition}

%%%%% remark %%%%%
\begin{example}[Regular case around closed leaf]\label{regular-case-1}
If $\F$ is regular and $B = L$ is a closed leaf of $\F$ then, by counting dimensions, we conclude that $\mathcal{N}$ needs to be the zero distribution and $\mathcal{T}$ needs to be the linearization of $T \F$. In this case the induced parcial connection $\nabla^{\tau}$ is in fact a (total) connection. More precisely $\nabla^{\tau}$ is the restriction to $L$ of the Bott connection of $\F$, since both have the same horizontal distribution.
\end{example}

The distributions $\widehat{\mathcal N}$ and $\mathcal N$ satisfy the following property:

%%%%% proposition %%%%%
\begin{proposition}[\cite{Alexandrino-Radeschi-Molinosconjecture}] \label{proposition-transverlinearfield}
For every smooth $\F_B$-basic vector field $\vec{Y}_{0}$ along a plaque $P$ in $B$ there exists a smooth extension $\vec{Y}_{0}$ to an open set of $U$ such that
\begin{enumerate}
\item $\vec{Y}_{0}$ is foliated and tangent to $\widehat{\mathcal N}$.

\item The linearization $\vec{Y}:=\vec{Y}_{0}^{\ell}$ of $\vec{Y}_{0}$ with respect to $B$ is tangent to $\mathcal{N}$, and it is foliated with respect to both $\F$ and $\Flin$.
\end{enumerate}
\end{proposition}

We also need a simple but important observation.

%%%%% corollary %%%%%
\begin{corollary} \label{corollary-isomorphismK}
Let $\vec{Y}$ be the vector field defined in Proposition \ref{proposition-transverlinearfield}, $\varphi^{Y}_{s}$ the local flow associated to $\vec{Y}$, and $K^{0}_{b}$ the Lie group defined in Section \ref{subsection-linearized foliation}. Then $\varphi_{s}^{Y}$ induces an isomorphism $\widehat{\varphi}: K^{0}_{b} \to K^{0}_{\tilde{b}}$ where $\tilde{b} = \varphi_{s}^{Y}(b)$.
\end{corollary}

\begin{proof}
First we claim that \emph{if a flow of a  vector field $\vec{Y}$ sends a vector field $\vec{X}^{1}$ to $\vec{X}^{2}$, then the flow of  linearization of  $\vec{Y}$ sends the linearization of $\vec{X}^{1}$ into the linearization of $\vec{X}^{2}$}. In fact, let $\hat{\varphi}_{s}$, $\varphi^{1}_{t}$ and $\varphi^{2}_{t}$ be the flows of the vector fields $\vec{Y}$, $\vec{X}^{1}$ and $\vec{X}^{2}$. From the hypothesis we have that $\varphi_{t}^{2} = \hat{\varphi}_{s} \circ \varphi^{1}_{t} \circ \hat{\varphi}_{-s}$. Note that the flows of $\vec{Y}_{\lambda}$, $\vec{X}^{i}_{\lambda}$ (for $i = 1, 2$) are $\hat{\varphi}_{s,\lambda} = h_{\lambda}^{-1} \circ \hat{\varphi}_{s} \circ h_{\lambda}$ and $\varphi_{t,\lambda}^{i} = h_{\lambda}^{-1} \circ \varphi_{t}^{i} \circ h_{\lambda}$ respectively. Therefore $\hat{\varphi}_{s,\lambda} \circ \varphi^{1}_{t,\lambda} \circ \hat{\varphi}_{-s,\lambda} = \varphi^{2}_{t,\lambda}$. The claim now follows by taking  the limit when $\lambda$ goes to zero.

Therefore if $t \to g_{t}$ is a  flow of a linearized foliated vector field on $U_b$ that preserves $\F_b$, i.e., $g_{t} \in K_{b}^{0}$, then $t \to \hat{\varphi}(g_{t}) := \varphi_{s}^{Y}|_{U_b} \circ g_{t} \circ \big(\varphi_{s}^{Y}|_{U_b}\big)^{-1}$ is also a flow of  a linearized foliated vector field on $U_{\tilde{b}}$ fixing $\F_{\tilde{b}}.$ Therefore $t \to \hat{\varphi}(g_{t})$ is a flow of a Killing vector field that preserves  $\F_{\tilde{b}}$, i.e., it is contained in $K_{\tilde{b}}^{0}$.
\end{proof}

In what follows we  say that a (partial) connection $\nabla^{\tau}$  compatible with the metric of $E=\nu(B)$ is a \emph{$\F$-compatible metric connection} 
if for each curve $\alpha\subset L\subset B$ and  
each parallel field $t\to\xi(t)\in \Gamma(\alpha^{*}E)$ along $\alpha$, we have that $\xi(t)\in L_{\xi(0)}.$

\begin{lemma}
\label{lemma-total-connection}
Assume that $B=\overline{L_q}$ i.e., $B$ is the closure
of a leaf $L_q$. Then there exists a $\overline{\F}$-compatible metric connection
$\nabla^{\overline{\tau}}:\mathfrak{X}(B)\times\Gamma(E)\to \Gamma(E)$  that extends
the $\F$-partial compatible metric connection $\nabla^{\tau}$.
\end{lemma}
\begin{proof}

Our goal  is to find a linear connection $\overline{\mathcal{T}}$ (and hence $TE= \mathcal{K}\oplus \overline{\mathcal{T}}$) 
so that:
\begin{enumerate}
\item[(a)] $\overline{\mathcal{T}}\subset T\overline{\F}$,
\item[(b)] $\mathcal{T}\subset\overline{\mathcal{T}}$.
\end{enumerate}
The fact that $\overline{\F}=\{\overline{L}\}$ is a S.R.F with closed leaves (see \cite{Alexandrino-Radeschi-Molinosconjecture}) and  item (a)  
will assure that $\overline{\mathcal{T}}$  
induces the $\overline{\F}$-compatible  metric connection $\nabla^{\overline{\tau}},$ (recall Example \ref{holonomy-foliation}).
Item (b) will imply that $\nabla^{\overline{\tau}}$ is an extension of the   $\F$-partial compatible metric connection $\nabla^{\tau}.$

Let us consider an open covering $\{U_{\alpha}^{b}\}$ of $B$ where $U_{\alpha}^{b}$ are precompact open sets of $B$. 
Define $U_{\alpha}=\mathrm{Tub}_{\epsilon}(U_{\alpha}^{b})=\pr^{-1}(U_{\alpha}^{b})$ where   $\pr: \mathrm{Tub}_{\epsilon}(B) \to B$ is the project metric. Therefore
$\mathrm{Tub}_{\epsilon}(B)=\cup_{\alpha} U_{\alpha}$ and $U_{\alpha}$ is homothetic invariant. 
We claim that \emph{ there exists a regular homothetic distribution  
$\mathcal{N}_\alpha$ on $U_\alpha$ so that
\begin{itemize}
\item $\mathcal{N}_{\alpha}\subset T\overline{\F},$
\item $TE= \mathcal{T}  \oplus \mathcal{N}_{\alpha} \oplus \mathcal{K},$
\item $\mathcal{N}_{\alpha}|_{B}=\nu(\F|_{B})$.
\end{itemize}}
In fact, let $\{ X_{\alpha}^{i} \} $ be  an orthonormal frame of $\nu(\F|_{U_\alpha})$.
By \cite{Alexandrino-Radeschi-Molinosconjecture}, we can extend these vector fields to vector fields tangent 
to $\overline{\F}$ and then linearize them to produce 
linear independent  vector fields  $\{(X_{\alpha}^{i})^{\ell}\}.$ This process
can be done at least in a smaller tubular neighborhood $\mathrm{Tub}_{\epsilon_0}(U_{\alpha}^{b})$
for $\epsilon_0<\epsilon$. But since linearized vector fields are homothetic invariant,
and the  homothetic transformation $h_\lambda$ are diffeomorphisms, 
 we can extend  $\{(X_{\alpha}^{i})^{\ell}\}$ to linearly independent vector fields
defined on  $\mathrm{Tub}_{\epsilon}(B)$
and set $\mathcal{N}_{\alpha}$ the distribution generated by these vector fields.
Since the homothetic transformation preserves the distrubtion $\mathcal{K}$, $\mathcal{T}$
and send closure leaves to closure leaves, we conclude that the $\mathcal{N}_\alpha$ 
satisfies the desired properties.   

Now we define, for each $\alpha$, a metric $g_\alpha$ on $U_\alpha$ so that:
\begin{enumerate}
\item[(i)]  $\mathcal{K}$ is orthogonal (with respect to $g_\alpha$) to the distribution
$\mathcal{T}\oplus \mathcal{N}_\alpha$ , which is a regular distribution contained in $T\overline{\F}$;
\item[(ii)] $g|_{\mathcal{K}}=g_{\alpha}|_{\mathcal{K}}.$
\end{enumerate}
Property (i) implies that $\nu(\overline{\F})$ is contained in $\mathcal{K}$ 
and property (ii)  that the vector space $\nu(\overline{L})$ does not depend on $\alpha$. 

Finally we define $\tilde{g}=\sum_{\alpha} \varphi_{\alpha} g_{\alpha},$
where $\{\varphi_\alpha\}$ is a partition of unity subordinated to $\{U_\alpha\}$. 
 Set $\widetilde{\mathcal{\overline{T}}}$ the orthogonal complement (with respect to $\tilde{g}$)
of $\mathcal{K}$. From property (i) and definition of $\tilde{g}$,
 we infer that $\mathcal{T}$ is $\tilde{g}$ orthogonal to $K$ and hence
\begin{equation}
  \label{eq-1-lemma-total-connection}
\mathcal{T}\subset \widetilde{\mathcal{\overline{T}}}. 
\end{equation}
As we remarked before,  
the normal distribution of $\overline{\F}$ does not depend on $\alpha$ and 
is contained in $\mathcal{K}$,  and hence  the normal distribution 
$\nu(\overline{\F})$ (with respect to $\tilde{g}$) is contained in $\mathcal{K}$.
The fact that $\nu(\overline{\F}) \subset \mathcal{K}$ and  $\mathcal{K}$ is the
orthogonal complement (with respect $\tilde{g}$) of 
$\widetilde{\mathcal{\overline{T}}}$ imply that 
\begin{equation}
\label{eq-2-lemma-total-connection}
\widetilde{\mathcal{\overline{T}}}\subset T\overline{\F}. 
\end{equation}
Set  $\mathcal{\overline{T}}:=(\widetilde{\mathcal{\overline{T}}})^{\ell}.$
Eq \eqref{eq-1-lemma-total-connection}, \eqref{eq-2-lemma-total-connection}
and the fact  that  $(T\overline{\F})^{\ell}\subset T\overline{\F}$
and $\mathcal{T}^{\ell}=\mathcal{T}$ imply 
$$\mathcal{T}=\mathcal{T}^{\ell}\subset (\widetilde{\mathcal{\overline{T}}})^{\ell}=
\overline{\mathcal{T}}=(\widetilde{\mathcal{\overline{T}}})^{\ell} \subset (T\overline{\F})^{\ell}\subset T(\overline{\F}) $$
and hence items (a) and (b) are fulfilled.

\end{proof}

\begin{remark}
In Lemma \ref{lemma-total-connection}, the distribution $\mathcal{N}$ of 
the triple $(\mathcal{K},\mathcal{T},\mathcal{N})$ associated to $\F$ 
can been chosen to be  a distribution  tangent 
to the leaves of $\overline{\F}$.  
It is defined as the distribution in $\overline{\mathcal{T}}$
so that $\overline{\mathcal{T}}=\mathcal{T}\oplus\mathcal{N}$
and  $\pr_{*}(\mathcal{N})=\nu(\F|_{B}).$
\end{remark}

%%%%%%%%%%%%%%% The local closure foliation %%%%%%%%%%%%%%%%%%%%%%%%%%%%%%%%%%%%%%%%%%%%
\subsubsection{The local closure foliation}

Let us now recall the construction of the foliation $\widehat{\F}^\ell$, that has the following properties: $\Flin\subset \widehat{\F}^\ell \subset \overline{\Flin}$ and the restriction $\widehat{\F}^\ell_b$ to each $\pr$-fiber $U_b$ is homogeneous and closed since it is formed by the orbits of $\overline{K_{b}^{0}}$. A leaf $\widehat{L}_q$ of $\widehat{\F}^\ell$ through $q$ is defined as:
\[
\widehat{L}_q = \left\{ \Phi(k \cdot q) \mid \Phi \in \mathsf{D}, \, k \in \overline{K_{b}^{0}} \right\}
\]
The foliations $\Flin$ and $\widehat{\F}^\ell$ are SRF with respect to the metric which turns $\mathcal{K}$ orthogonal to $\mathcal{T} \oplus \mathcal{N}$, preserving the metric of each component, i.e., the flat metric induced by exponential map in $\mathcal{K}$ and the metric $\pr^\ast g_B$ in $\mathcal{N} \oplus \mathcal{T}$ where $g_B$ is the restriction of $g$ on $B$ (see figure \ref{TNK}).

More precisely $\widehat{\F}^\ell$ is an \emph{orbit like foliation}, that is, a SRF such that each infinitesimal foliation is a SRF given by orbits of a compact subgroup of $O(n)$ (see \cite{Alexandrino-Radeschi-Molinosconjecture} for the definition and properties).

\begin{example}
An illustration of the concept of orbit-like foliation can be extracted from example \ref{holonomy-foliation}. Consider in this example a holonomy foliation $\F^\tau$ determined by a linear connection $\nabla^\tau: \mathfrak{X}(B) \times \Gamma(E) \to \Gamma(E)$ such that the holonomy groups are all compact (e.g. a Riemannian connection in a Riemannian manifold and a normal connection of a submanifold embedded in a Euclidean space as presented in \cite{Berndt-Console-Olmos}). Since the intersection of the leaves of $\F^\tau $ with the fibers of $E$ are given by orbits of the holonomy groups, in this particular case, the holonomy foliation already is an orbit-like foliation from which follows that $\F^\tau = (\F^\tau)^\ell = (\widehat{\F^\tau})^\ell$.
\end{example}

%%%%%%%%%%%%%%% A Few Facts About Lie Groupoids %%%%%%%%%%%%%%%
\subsection{A few facts about Lie groupoids} \label{section-facts-groupoids}

Recall that a Lie groupoid is composed of two manifolds $\mathcal{G}_1$ and $\mathcal{G}_0$ where elements of $\mathcal{G}_1$ are thought of as arrows between elements of $\mathcal{G}_0$. The maps $s,t: \mathcal{G}_1 \to \mathcal{G}_0$ which associate to an arrow $g \in \mathcal{G}_1$ its source and its targets are required to be a surjective submersions. It then follows that the space
\[\mathcal{G}_2 = \{(g,h) \in \mathcal{G}_1\times \mathcal{G}_1: s(g)=t(h)\}\]
of composable arrows is a manifold and there is a smooth multiplication map
\[m: \mathcal{G}_2 \longrightarrow \mathcal{G}_1, \quad m(g,h) = gh\]
which is associative and satisfies $s(gh) = s(h)$ and $t(gh) = t(g)$. A Lie groupoid $\mathcal{G}_1$ also comes equipped with a smooth embedding of $\mathcal{G}_0$ into $\mathcal{G}_1$
\[u: \mathcal{G}_0 \to \mathcal{G}_1, \quad u(x) = 1_x\]
which allows us to view each $x \in M$ as an identity arrow $1_x$ whose source and target is $x$. Needless to say, this arrow acts as an identity:
\[1_xh = h, \text{ and } g1_x = g \text{ for all } h \in t^{-1}(x), \text{ and } g \in s^{-1}(x).\]
Finally, there is a diffeomorphism $i: \mathcal{G}_1 \to \mathcal{G}_1$ which associates to each arrow $g \in \mathcal{G}_1$ its inverse arrow $i(g) = g^{-1}$ which satisfies
\[s(g^{-1}) = t(g),\quad t(g^{-1}) = s(g), \quad gg^{-1} = 1_{t(g)}, \quad g^{-1}g = 1_{s(g)}.\]
We will denote a Lie groupoid by $\mathcal{G}_1 \rightrightarrows \mathcal{G}_0$.

Any Lie groupoid $\mathcal{G} = \mathcal{G}_1 \rightrightarrows \mathcal{G}_0$ induces a foliation on  $\mathcal{G}_0$ whose leaves are the connected components of  \emph{the orbits of} $\mathcal{G}$, $\mathcal{O}_x = \{t(s^{-1}(x))\}$.

Some examples of Lie groupoids will be important for us in this paper. We present them here.

%%%%% example %%%%%
\begin{example}
An example of a Lie groupoid that will be important throughout this paper is the \emph{Holonomy groupoid} 
$\mathrm{Hol}(\F) \rightrightarrows M$ of a regular foliation $\F$ on $M$. 
An arrow in $\mathrm{Hol}(\F)$ is a class of a path in a leaf of $\F$, where the equivalence relation identifies leafwise homotopic paths and paths inducing the same germ of diffeomorphisms sliding transversals along the paths.
The source of $[\alpha] \in \mathrm{Hol}(\F)$ is $\alpha(0)$, the starting point of $\alpha$, and the target of $[\alpha]$ is its endpoint $\alpha(1)$. 
Multiplication is given by concatenation  of paths. 
The identity arrow at $x \in M$ is the homotopy class of the constant path at $x$, and inversion is given by
\[[\alpha]^{-1} = [\bar{\alpha}], \text{ where } \bar{\alpha}(t) = \alpha(1-t).\]
The orbits of $\mathrm{Hol}(\F)$ are precisely the leaves of $\F$.
\end{example}

%%%%% example %%%%%
\begin{example}
A Lie groupoid over a point is just a Lie group. More generally, a bundle of Lie groups is just a Lie groupoid for which $s = t$. In this case, each $s$-fiber inherits the structure of a Lie group and we can view $\mathcal{G}_1$ as a smooth family of Lie groups parameterized by $\mathcal{G}_0$. 
\end{example}

\begin{example}\label{example:action-groupoid}
Let $\mu:G\times M\to M$ be an action. Then the \emph{action groupoid or transformation groupoid} 
is defined as $\mathcal{G}_1=G\times M$ and $\mathcal{G}_0=M$
with source map $s(g,x)=x$, target map $t(g,x)=\mu(g,x)$ and unity map $u(x)=(e,x)$. The product map is the composition, i.e,
$m\big((h_2,y),(h_1,x)\big)=(h_{2}h_{1},x)$ and the inverse map is $i(h_1,x)=(h_{1}^{-1},\mu(h_{1},x)).$ 
\end{example}

\begin{example}
Associated to each $G$-principal bundle $\pi: P\to B$ there exists a transitive Lie groupoid with isotropy equals to $G$ called the \emph{gauge groupoid} of $P$.
This groupoid can be realized as the quotient of $P\times P$ by the diagonal action of $G$, i.e., $\mathcal{G}_1:=(P\times P)/G$ and $\mathcal{G}_0:=B$ with structure determined by
\[s([p,q]) = \pi(q), \quad t([p,q]) = \pi(p), \quad [p,q]\cdot[q,r] = [p,r].\]
\end{example}

\begin{example}\label{ex:gl-groupoid}
Similar to vector spaces, a vector bundle $E\to B$ has a \emph{general linear groupoid} $GL(E)\rightrightarrows B$
whose arrows are the linear isomorphisms between the fibers.
This groupoid can be realized as the gauge groupoid of the frame bundle $F(E)\to B$.

In this paper we usually assume that $E$ is an Euclidean vector bundle. In this case we reduce the general linear groupoid of $E$ to obtain the \emph{orthogonal linear groupoid} $\mathcal{O}(E)\rightrightarrows M$, whose arrows are the linear isometries between the fibers of $E$. Equivalently, $\mathcal{O}(E)$ can be realized as the gauge groupoid associated to the orthogonal frame bundle of $E$. 
\end{example}

In this paper we will be interested in applying two general constructions for Lie groupoids to the specific setting coming from the study of SRFs, namely, we will need to take the quotient of a Lie groupoid by a free and proper action of a Lie group $G$ through automorphisms, and we will need to consider the transformation groupoid associated to a representation of a Lie groupoid on a vector bundle. We now explain these constructions.

Let $G$ be a Lie group and $\mathcal{G} = \mathcal{G}_1 \rightrightarrows \mathcal{G}_0$ be a Lie groupoid. An \emph{action of $G$ on $\mathcal{G}$ by Lie groupoid automorphisms} is an action of $G$ on $\mathcal{G}_1$ and on $\mathcal{G}_0$ such that for each $a \in G$ the map
\[\xymatrix{
\mathcal{G}_1 \ar@<0.25pc>[d] \ar@<-0.25pc>[d] \ar[rr]^{\Psi_a} & & \mathcal{G}_1 \ar@<0.25pc>[d] \ar@<-0.25pc>[d] \\
\mathcal{G}_0 \ar[rr]_{\psi_a} & & \mathcal{G}_0}
\]
is a Lie groupoid morphism, i.e., commutes with all of the structure maps. The action is said to be \emph{free and proper} if the action on $\mathcal{G}_1$ is free and proper. We remark that since $s: \mathcal{G}_1 \to \mathcal{G}_0$ is a surjective submersion with a $G$-invariant global section $u: \mathcal{G}_0 \to \mathcal{G}_1$, it follows that the action of $G$ on $\mathcal{G}_0$ is also free and proper.

In this paper the action that will appear naturally is a right action.

%%%%% propostion %%%%%
\begin{proposition} \label{prop:quotient}
Let $G$ be a Lie group which acts freely and properly on a Lie groupoid $\mathcal{G}_1 \rightrightarrows \mathcal{G}_0$ through automorphisms. Then $\mathcal{G}_1/G \rightrightarrows \mathcal{G}_0/G$ has an induced structure of a Lie groupoid.
\end{proposition}

\begin{proof}
We define the source and target of $\mathcal{G}/G$ in the obvious way:
\[\bar{s}[g] = [s(g)], \quad \bar{t}[g] = [t(g)].\]
The maps $\bar{s}$ and $\bar{t}$ are well defined because of the $G$-equivariance of $s$ and $t$. Moreover, they are smooth surjective submersions because $s$, $t$ and the projection $p: \mathcal{G} \to \mathcal{G}/G$ are smooth surjective submersions.

Next we define the multiplication of $\mathcal{G}/G$. A pair $([g], [h])$ of arrows of $\mathcal{G}/G$ is composable iff $[s(g)] = [t(h)]$. Therefore, there exists a unique $a \in G$ such that $s(g) = t(h)a = t(ha)$. We define
\[[g][h] = [g(ha)].\]
The reader can easily check that the multiplication map is well defined and associative.

Moreover, we note that the multiplication $\bar{m}: (\mathcal{G}/G)_2 \to \mathcal{G}/G$ sits in the following commutative diagram
\[\xymatrix{
\mathcal{G}_2 \ar[d]_{p_2} \ar[r]^m & \mathcal{G}\ar[d]_p \\
(\mathcal{G}/G)_2 \ar[r]_{\bar{m}} & \mathcal{G}/G,}
\]
where
\[p_2: \mathcal{G}_2 \to (\mathcal{G}/G)_2, \quad p_2(g,h) = ([g],[h])\]
is a surjective submersion. It follows that $\bar{m}$ is smooth.

The other structure are defined similarly and are clearly smooth. They are defined by
\[[g]^{-1} = [g^{-1}], \quad \bar{u}([x]) = [1_x].\]

Finally, we note that the unit and inverses of the quotient groupoid indeed satisfy the properties in the definition of a Lie groupoid. For example, if $\bar{s}([g]) = [x]$, then $s(g) = xa$ and
\[[g][1_x] =[g (1_x a)] = [g1_{xa}] = [g].\]
The other properties are proven similarly by direct computations.
\end{proof}

We next describe the transformation groupoid associated to a representation of a Lie groupoid $\mathcal{G}_1\rightrightarrows \mathcal{G}_0$ on a vector bundle $\pi:E \to \mathcal{G}_0$. Recall that a \emph{representation of $\mathcal{G}_1\rightrightarrows \mathcal{G}_0$ on a vector bundle $\pi: E \to \mathcal{G}_0$} is a smooth map
\[\psi: \mathcal{G}_1 \times_{\mathcal{G}_0} E \longrightarrow E \quad \psi(g,e) = ge,\]
where $\mathcal{G}_1 \times_{\mathcal{G}_0} E = \{(g, v) \in \mathcal{G}_1 \times E : s(g) = \pi(v)\},$ satisfying the following properties: 
\begin{itemize}
\item $\psi(g, \cdot) = \psi_g: E_{s(g)} \to E_{t(g)}$ is a linear isomorphism for all $g \in \mathcal{G}_1$;
\item $1_{\pi(v)}v = v$ for all $v \in E$;
\item $g(hv) = (gh)v$ for all $(g,h) \in \mathcal{G}_2$ and all $v \in E_{s(h)}$.
\end{itemize}
Equivalently, a representation of $\mathcal{G}$ on $E$ can be recast a Lie groupoid morphism $\mathcal{G} \to \mathrm{GL}(E)$.

Given a representation of $\mathcal{G}_1\rightrightarrows \mathcal{G}_0$ on a vector bundle $\pi: E \to \mathcal{G}_0$, one obtains a new groupoid $\mathcal{G} \ltimes E$ called the \emph{transformation Lie groupoid of the representation}. Its structure is described bellow.

The manifolds of arrows and objects are
\[(\mathcal{G} \ltimes E)_1 = \mathcal{G}_1 \times_{\mathcal{G}_0} E, \quad (\mathcal{G} \ltimes E)_0 = E. \]
Its source, target and multiplication is given by
\[s(g,v) = v, \quad t(g,v) = gv, \quad (g, hv)(h,v) = (gh, v).\]
Finally, its unit and inverse is given by
\[1_v = (1_{\pi(v)},v), \quad (g,v)^{-1} = (g^{-1}, gv).\]
An example relevant to the theory of SRFs will be given in the next section.

\begin{remark}
For any saturated submanifold $B$ of $\mathcal{G}_0$, there is a canonical representation of the restriction groupoid $\mathcal{G}_B \rightrightarrows B$ on the normal bundle $\nu(B)$ called \emph{normal representation}.
The transformation groupoid of the normal representation is a \emph{local linear model} for $\mathcal{G}$ around $B$.
There are linearization results identifying $\mathcal{G}$ around $B$ with the local linear model, for instance if $\mathcal{G}$ is proper \cite{Crainic-Struchiner,Fernandes-Hoyo}.
Going back to SRFs, we will show in Proposition \ref{proposition-groupoid-linear-foliation} that the linearized foliation $\Flin$ is the orbit foliation of a representation, and this may be interpreted as the local linear model for a possibly groupoid presenting $\F$.
\end{remark}

\subsection{A few facts about Lie algebroids} \label{section-facts-algebroids}
In Section \ref{section-rotate-translate-groupoid} we will also make use of the infinitesimal object associated to a Lie groupoid, known as a \emph{Lie algebroid},  and some elements of its integration theory which we recall here. A Lie algebroid is a vector bundle $\pi: A \to M$ endowed with:
\begin{itemize}
\item a Lie bracket $[\cdot,\cdot]$ on the space $\Gamma(A)$ of sections of $A$;
\item a vector bundle map $\rho: A \to TM$ known as the \emph{anchor of the Lie algebroid},
\end{itemize}
such that the following Leibniz identity holds for all sections $\alpha, \beta \in \Gamma(A)$ and for all smooth maps $f \in \mathrm{C}^\infty(M)$
\[[\alpha, f\beta] = f[\alpha, \beta] + \rho(\alpha)(f)\beta.\]

Every Lie groupoid $\mathcal{G}_1 \rightrightarrows \mathcal{G}_0$ determines a Lie algebroid $A \to \mathcal{G}_0$. The construction of $A$ is analogous to the construction of the Lie algebra of a Lie group, but taking into account that on a Lie groupoid there are many identity elements, and that right translation by an element $g \in \mathcal{G}_1$ is only defined on the source fiber of $t(g)$. Explicitly, one takes the vector bundle $A \to \mathcal{G}_0$ to be the pullback by $u$ of the kernel of $ds$. It then follows that the sections of $A$ identify with the space of right invariant vector fields on $\mathcal{G}_1$ and this identification induces a Lie bracket on $\Gamma(A)$. The anchor map is given by the restriction of $dt$ to $A$. A simple verification shows that one obtains in this way a Lie algebroid out of a Lie groupoid.

\begin{example}
%%%
%%%%
Let $\mu:G\times M\to M$ be an action and let $\mathcal{G}_1=G\times M \rightrightarrows M = \mathcal{G}_0$ be the action groupoid of Example \ref{example:action-groupoid}. Following  the previous discussion we conclude that the associated
Lie algebroid $A\to \mathcal{G}_0$, must be $A=\mathfrak{g}\times M\to M$. Here it is clear that Lie bracket on $\Gamma(A)$
is induced by  right invariant vector fields on $G$ and their infinitesimal action on $M$. Moreover, the anchor map $d t:A\to TM$ defined as $ dt (\phi_x)=d\mu_x \phi=\frac{d}{d t}\mu(\exp(t\phi),x)|_{t=0}$
induces a Lie algebra morphism between $\Gamma(A)$ with the  \emph{fundamental vector fields} $x\to\phi^{\#}(x):=d\mu_x \phi \in \Gamma(TM)$,
see Figure \ref{holonomy}. %\ref{figure-liealgebroid}. 
\end{example}

%%%%% Figura %%%%%
\begin{figure}[!htb]
%\label{figure-liealgebroid}
\centering
\includegraphics[width=0.9\textwidth]{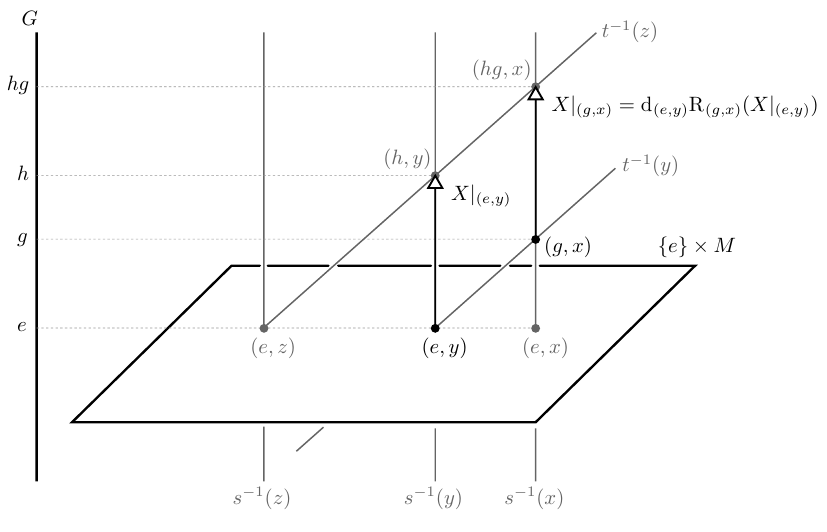}
\caption{An illustration of the action groupoid and its respectively Lie algebroid. Here $z = h g \, x$, $y = g \, x$ and $X \in \Gamma(A)$.}
\label{holonomy}
\end{figure}

Here are some other examples of Lie algebroids which will be relevant for us in this paper:

\begin{example}
The tangent distribution of a regular foliation $\F$ gives rise to a subbundle $T\F$ of $TM$.
Since this distribution is involutive the bracket of sections of $T\F$ is again a section of $T\F$. This bracket, together with the inclusions $T\F \to TM$ as an anchor form a structure of Lie algebroid to $T\F$.
\end{example}

\begin{example}
Another example of a Lie algebroid that will be used in this paper is that of a bundle of algebras.
A bundle of Lie algebras is just a Lie algebroid $A\to M$ for which the anchor satisfies $\rho\equiv 0$.
In this case, each fiber inherits the structure of a Lie algebra and we can view $A$
as a smooth family of Lie groups parameterized by $M$.
\end{example}

\begin{example}
The Lie algebroid of the general linear groupoid of a vector bundle $E \to M$ (Example \ref{ex:gl-groupoid}) is called the \emph{general linear algebroid} of $E$, and is denoted by  $\mathfrak{gl}(E)$.

As a vector bundle, $\mathfrak{gl}(E)$ fits into a short exact sequence
\begin{equation}\label{eq:exact-sequence}0\longrightarrow E^{*}\otimes E\longrightarrow \mathfrak{gl}(E) \longrightarrow TM \longrightarrow 0.\end{equation}
Its space of sections can be identified with the space of degree 1 derivations of the vector bundle $E$, i.e., the space of  linear operator $D:\Gamma(E)\to \Gamma(E)$ such that there exists a vector field $X_D$ in $M$, satisfying $D(fs)=X_D(f)s + fD(s)$ for all sections $s \in \Gamma(E)$ and functions $f \in \mathrm{C}^\infty(M)$. The Lie bracket os two derivations is the commutator bracket.

We observe that the vector bundle splittings of the exact sequence \eqref{eq:exact-sequence} are in 1-1 correspondence with linear connections on $E$. In fact, a connection $\nabla$ produces the splitting $\sigma(X)=\nabla_X$. When $E$ is an Euclidean vector bundle, an analogous construction can be made so that the splittings correspond bijectively to connections compatible with the fiberwise metric. 

The general linear algebroid of $E$ can also be realized as the \emph{Atiyah algebroid} of the frame bundle of $E$. We recall that the Atiyah algebroid of a principal $G$-bundle $\pi:P \to M$ is $A = \frac{TP}{G}$ as a vector bundle (over $M$). Its Lie bracket on the space of sections is obtained by identifying sections of $A$ with $G$-invariant vector fields on $P$, and its anchor is induced by $d\pi: TP \to TM$.

For more details see \cite{Crainic-Moerdijk}.
\end{example}

A Lie algebroid is called \emph{integrable} if it is isomorphic to the Lie algebroid of a Lie groupoid. In contrast to the usual Lie theory for Lie groups, not every Lie algebroid is integrable, but the obstructions for integrability are well known (see \cite{Crainic-Fernandes}). In this paper we will only need the fact that every Lie subalgebroid of an integrable Lie algebroid is itself integrable. Even though this result follows from the general obstruction theorem of \cite{Crainic-Fernandes}, we will use here the approach of \cite{Moerdijk-Mrcun-subalgebroid} where an explicit description of the integrating groupoid is given as follows.

Let $A \to M$ be a Lie subalgebroid of $A' \to M$, and $\mathcal{G} = \mathcal{G}_1\rightrightarrows \mathcal{G}_0 = M$ be a Lie groupoid which integrates $A'$. Then for each $g \in \mathcal{G}_1$ we may consider the subspace of $\ker ds$ obtained by right translating $A_t(g) \subset \ker d_{1_{t(g)}}s$ to $\ker d_gs$. In this way, one obtains a $s$-vertical involutive  distribution on $T\mathcal{G}_1$, i.e., a regular foliation $\F^A$ on $\mathcal{G}_1$. It then follows that the holonomy groupoid of this foliation is a Lie groupoid $\mathrm{Hol}(\F^A) \rightrightarrows \mathcal{G}_1$. Moreover, $\mathcal{G}$ acts on  $\mathrm{Hol}(\F^A)$ by automorphisms and the quotient $\mathcal{H} \rightrightarrows \mathcal{G}_0$ is a Lie groupoid integrating $A$. This Lie groupoid comes equipped with a Lie groupoid immersion $i: \mathcal{H} \to \mathcal{G}$ whose derivative restricts to the inclusion of $A$ into $A'$. For further details we refer to \cite{Moerdijk-Mrcun-subalgebroid}.

\begin{remark}
It is also proven in \cite{Moerdijk-Mrcun-subalgebroid} that the groupoid morphism $i: \mathcal{H} \to \mathcal{G}$ is injective if and only if the holonomy of the foliation $\F^A$ is trivial. 
\end{remark}

%%%%%%%%%%%%%%%%%%%%%%%%%%%%%%%%%%%%%%%%%%%%%%%%%%%%%%%%%%%%%%%%%%%%5

%%%%%%%%%%%%%%%%%%%%%%%%%%%%%%%%%%%%%%%%%%%%%%%%%%%%%%%%
%%%%%%%%%%%%%%% SEMI-LOCAL MODELS OF SRF %%%%%%%%%%%%%%%
%%%%%%%%%%%%%%%%%%%%%%%%%%%%%%%%%%%%%%%%%%%%%%%%%%%%%%%%
\section{Semi-local Models of SRF} \label{section-semi-local-modelsofSRF}

In this section we prove Theorem \ref{theorem-semi-local-model}. Using the results in Section \ref{section-facts-SRF} it suffices to prove the proposition below. Before we do so, let us stress some notation. First, denote the generalized holonomy foliation defined in Example \ref{generalized-holonomy-foliation} by $\F (\nabla^{\tau},\F^{E})$. Now remember that given $(M, g, \F)$ a SRF and $B \subset M$ a closed saturated submanifold contained in a stratum, it is possible to consider the linearized foliation $\Flin$ on a $\epsilon$-tubular neighborhood $U$ around $B$ and for each $b \in B$ we denoted by $\F_b$ the infinitesimal foliation on $\nu_b B$ obtained by homothetic extension of the foliation $(\exp^\nu)^{-1}(\F \cap U_b)$. By an abuse of notation we will be denoted by $\Flin$ the homothetic extension of $(\exp^{\nu})^{-1}(\Flin)$ on $\nu B$.

%%%%% proposition %%%%%
\begin{proposition} \label{lemma-linearized-vector-linearfoliation}
Let $(M, g, \F)$ be a SRF, $B \subset M$ a closed saturated submanifold contained in a stratum, $\mathcal{T}$ the distribution on $E = \nu B$ described in section \ref{subsection-3-distribution} and $\F^{E}$ the singular foliation on $E$ fiberwise determined by the infinitesimal foliations of $\F$. Then 
\begin{enumerate}
\item[(a)] $\F (\nabla^{\tau}, \F^{E}) = \F$.
\item[(b)]$\F (\nabla^{\tau}, \F^{E})^\ell = \Flin$.
\end{enumerate}
\end{proposition}

\begin{proof}

\

\begin{enumerate}

\item[(a)] 
%%%%%%%%%%%%%%%%%%%%

$\F (\nabla^{\tau}, \F^{E}) = \F$ follows direct from the definition of 
$\F^{E}$ (i.e., $\F^{E}$ is the foliation which leaves are leaves of $\F_{b}$, $\forall b\in B$)
and the fact that $\nabla^{\tau}$ is a  $\F$-invariant
connection (i.e.,  for any $v\in E$
and any $\F_{B}$-leafwise path $\alpha$ the 
$\mathcal{T}$-horizontal curve $t\to \mathcal{P}_{\alpha}^{t}(v)$
is contained in $L_v$).

%%%%%%%%%
\item[(b)] Since the distribution $\mathcal{T}$ is tangent to $\Flin$ and the orbits of $K_{b}^{0}$ are contained in $\Flin$ we conclude that $\F (\nabla^{\tau}, \F^{E})^\ell \subset \Flin$.

Now we want to prove that $\Flin \subset \F (\nabla^{\tau}, \F^{E})^\ell$, i.e., that flows of linearized vector fields are contained in $\F (\nabla^{\tau}, \F^{E})^\ell$.

Let $\varphi_{t}$ be the flow of a linearized vector field tangent to $\F$, $b_0 \in B$, $v_{0} \in E_{b_0}$ and let $P_{b_0}$ be a plaque through $b_0$ in a normal (geodesic) coordinate system, so that for each $b \in P_{b_0}$ one can associate a unique geodesic segment $\gamma_{b_0, b} \subset P_{b_0}$ connecting $b$ to $b_0$. Consider $\mathcal{P}_{\gamma_{b_0, b_t}}: E_{b_t} \to E_{b_0}$ the parallel transport along $\gamma_{b_0, b_t}$ where $b_t = \pi (\varphi_{t}(v_0))$ and $\pi: E|_{P_{b_0}} \to P_{b_0}$ the base point projection.

Then $t \to \mathcal{P}_{\gamma_{b_0, b_t}} \circ \varphi_{t} =: k_t$ is a curve of isometries of $E_{b_0}$ that fixes $\F_{b_0}$ and such that $k_0 = \text{Id}$. Therefore $t \to k_t$ is a curve in $K_{b_0}^{0}$ starting at the identity. It follows that $\varphi_{t}(v_0) = \mathcal{P}_{(\gamma_{b_0, b_t})^{-1}} (k_t (x_0)) \in \F (\nabla^{\tau}, \F^{E})^\ell$. This concludes the proof.
\end{enumerate}
\end{proof}

%%%%%%%%%%%%%%%%%%%%%%

%%%%%%%%%%%%%%%%%%%%%%%%%%%%%%%%%%%%%%%%%%%%%%%%%%%%%%
%%%%%%%%%%%%%%% LIE GROUPOID STRUCTURE %%%%%%%%%%%%%%%
%%%%%%%%%%%%%%%%%%%%%%%%%%%%%%%%%%%%%%%%%%%%%%%%%%%%%%
\section{Lie Groupoid Structures} \label{section-lie-groupoid-structure}

In this section we expose and discuss the Lie groupoid whose orbits are the leaves of the foliations $\F^{\ell}$ and $\widehat{\F}^\ell$, and in particular we prove Theorem \ref{theorem-linear-holonomy-groupoid} (this is done in  Theorem \ref{proposition-groupoid-linear-foliation} bellow). After the proof we exemplify our construction in two extreme cases: regular foliations around a closed leaf, and around a fixed point of a singular foliation. We hope that these examples will help the reader in understanding the main theorem.
We also presented here the proof of Theorem \ref{foliated-Ambrose-and-Singer}.

%%%%%%%%%%%%%%% Lie Groupoid Structure of Fl %%%%%%%%%%%%%%%
\subsection{The Holonomy Groupoid of $\Flin$}

%%%%% proposition %%%%%
\begin{theorem} \label{proposition-groupoid-linear-foliation}
Let $\F$ be a singular Riemannian foliation on a complete manifold $(M,g)$, $B$ be a closed saturated submanifold contained in a stratum and $U$ a saturated $\epsilon$-tubular neighbourhood of $B$. Then the leaves of the foliations  $\Flin$ and $\widehat{\F}^\ell$ are orbits of a Lie groupoid.
\end{theorem}

\begin{proof}
Let us prove that the leaves of $\Flin$ are orbits of a Lie groupoid. A similar proof holds for the foliation $\widehat{\F}^\ell$.

We denote by $E$ the normal bundle $\nu(B)$ of $B$. Set $\F^{\ell}_{0}$ the homothetic extensions of $(\exp^{\nu})^{-1}(\F^{\ell})$ to $E$. Recall that for each linearized flow on $U$ tangent to $\F$  we can associate a flow $t\to \varphi_{t}$ on $E$ so that $\varphi_{t}: E_{b} \to E_{\varphi_{t}(b)}$ is an isometry (which will also be called the \emph{linearized flow}) for each $b \in B$ where $\varphi_t(b)$ makes sense. The singular foliation $\F^{\ell}_{0}$ are the orbits of the pseudo-group generated by these flows.

Let $O(E)$ be the orthogonal frame bundle associated to $E\to B$. 
Note that each linearized flow $t\to \varphi_{t}$ on $E$ induces a flow $t \to \varphi_{t}$ on $O (E)$. 
Let $\widetilde{\F} = \{\widetilde{L}_{\xi_b}\}_{\xi_b \in O (E)}$ be the singular foliation obtained by compositions of these lifted flows. 
Let $\mathcal{H}^{\tau}$ be the horizontal distribution on $O(E)$ along $T \F_B$ induced by the partial linear connection $\nabla^{\tau}$ described in Proposition \ref{proposition-connections-tau-affine}.
Then for each frame $\xi_{b}\in O(E)_{b}$ we have:
\begin{equation}
\label{eq-1-lemma-groupoid-linear-foliation}
T_{\xi_{b}}\widetilde{L}_{\xi_{b}}=\mathcal{H}^{\tau}_{\xi_{b}}\oplus T_{\xi_{b}}\big( O(E)_{b}\cap  \widetilde{L}_{\xi_{b}}  \big)
\end{equation}
Since for each $b\in B$ the group $K_{b}^{0}$ (defined in Section \ref{preliminaries-foliation}) acts effectively on $E_{b}$, the group $K_{b}^{0}$ induces a free action on $O(E)_{b}$ and the orbits of this action coincide with the intersection of the leaves of $\widetilde{\F}$ with $O(E)_{b}$. In particular:
\begin{equation}
\label{eq-2-lemma-groupoid-linear-foliation}
\dim T_{\xi_{b}}\big( O(E)_{b}\cap  \widetilde{L}_{\xi_{b}}  \big)=\dim K_{b}^{0}.
\end{equation}
Once $ \dim K_{b}^{0}$ does not depend on $b\in B$ (recall Corollary \ref{corollary-isomorphismK}) we infer from equations \eqref{eq-1-lemma-groupoid-linear-foliation} and \eqref{eq-2-lemma-groupoid-linear-foliation} that the foliation $\widetilde{\F}$ is a (regular) foliation, i.e., the dimensions of the leaves are constant.

Let $\mathrm{Hol}(\widetilde{\F}) \rightrightarrows O(E)$ be the holonomy groupoid of the foliation $\widetilde{\F}$.
Note that the lifted flows act on $O(E)$ through bundle automorphisms, and, as such, they are $O(n)$-equivariant. 
It follows that the foliation $\widetilde{\F}$ is invariant by the right $O(n)$-action on $O(E)$.
This action maps leafwise curves to leafwise curves, and submanifolds of leaves to transversal submanifolds of leaves. Therefore, the action can be lifted to an action on the holonomy groupoid and it is easy to check that this gives a a free and proper action by automorphims on $\mathrm{Hol}(\widetilde{F})\rightrightarrows O(E)$.

By taking the quotient under the $O(n)$-action we obtain the following commutative diagram
\[\xymatrix{
\mathrm{Hol}(\widetilde{\F}) \ar@<0.25pc>[r] \ar@<-0.25pc>[r] \ar[d] & O(E) \ar[d] \\
\mathrm{Hol}(\widetilde{\F})/O(n) \ar@<0.25pc>[r] \ar@<-0.25pc>[r] & O(E)/O(n) = B.
}\]
We conclude that the groupoid $\mathcal{G} = \mathrm{Hol}(\widetilde{\F})/O(n) \rightrightarrows B$ is a Lie groupoid (recall Proposition \ref{prop:quotient}).

We remark that $\mathcal{G}$ comes with a canonical representation on $E$. If we identify $E$ with the associated bundle $O(E)\times_{O(n)} \mathbb{R}^n$, then the representation is given by
\[\bar{g}\cdot [\xi, v] := [t(g),v],\] 
where $g$ is the unique representative of $\bar{g}$ in $\mathrm{Hol}(\widetilde{\F}$ such that $s(g) = \xi$.
  
Note also that the orbits of this representation coincide with the leaves of $\F^\ell$. In fact, any leafwise curve of $\widetilde{\F}$ is homotopic to a concatenation of linearized flows, and we can represent the arrows of $\mathrm{Hol}({\widetilde{\F}})$ as concatenation of linearized flows, then the arrows of the representation $\mathcal{G}\ltimes E$ are products between arrows of the form $[(\tilde{\varphi}_{t},\xi, v)]$ with source $[(\xi,v)]$ and target $[(\varphi_t(\xi),v)]=\varphi_t([(\xi,v)])$ generated by linearized flows.
Therefore, by defining $\mathcal{G}^{\ell}=\mathcal{G}\ltimes E$ the transformation Lie groupoid of the representation we obtain a Lie groupoid whose orbits are the leaves of the foliation $\Flin$.

A similar construction holds for $\hat{\F}^\ell$.
\end{proof}

The Lie groupoid $\mathcal{G}^\ell \rightrightarrows E$ constructed in the previous theorem will be called the \emph{Linear Holonomy Groupoid of $\F$ around $B$}. 

\begin{remark}
We remark that the Lie groupoid $\mathcal{G} \rightrightarrows B$ comes with canonical representation on the normal bundle of $\F$ restricted to $B$. In fact, even though the normal bundle $\nu(\F) = TM/T\F$ is not a smooth vector bundle, when we restrict it to $B$ we obtain an honest vector bundle $\nu(\F)|_B = E \oplus \nu(\F_B)$, where $\nu(\F_B) = TB/T\F_B$. On the one hand, $\mathcal{G}$ has a representation on $E$ discussed in the previous theorem: it is the action which gives rise to the linearized foliation on $E$. On the other hand, since $\mathcal{G}$ is a source connected regular Lie groupoid with orbit foliations $\F_B$, it follows that there exists a morphism of Lie groupoids which is a surjective submersion $\mathcal{G} \to \mathrm{Hol}(\F_B)$ (see \cite{Crainic-Moerdijk-foliation}). By composing this morphism with the canonical action of $\mathrm{Hol}(\F_B)$ on $\nu(\F_B)$ we obtain a representation of $\mathcal{G}$ on $\nu(\F_B)$.   
\end{remark}

\begin{remark}[Monodromy groupoid] \label{rmk:monodromy}
We note that a similar construction can be made using the monodromy groupoid $\Pi_1(\widetilde{\F}) \rightrightarrows O(E)$ of the lifted foliation $\widetilde{\F}$ instead of the monodromy groupoid. 
The groupoid obtained after taking the quotient by the $O(n)$-action has the same leaves as the groupoid constructed in the proof above. 
In fact, there is a natural groupoid covering map $\Pi_1(\widetilde{\F})/O(n) \to \mathrm{Hol}(\widetilde{\F})/O(n)$ induced by the covering map existent before taking the quotient by the $O(n)$-action.

\end{remark}

%%%%% remark %%%%%
\begin{remark}[Regular case]\label{ex:regular-case}
In \cite{Molino}, Molino studied the structure of a (regular) Riemannian foliation $\F$ on a complete Riemannian manifold $B$ considering its lift to the associated orthogonal frame bundle. His construction could be considered a particular case of the construction above. More precisely in this case we can consider the vector bundle $E$ as the normal bundle of the foliation $\nu (\F)$ with the induced metric and the partial Bott connection $\nabla^{\tau}$ on $\nu(\F)$, where
\[\nabla^{\tau}_X (Y \text{ mod } \F) = [X,Y] \text{ mod } \F.\]
Since the Bott connection is locally flat, one sees that the lifted foliation $\widetilde{\F}$ on $O(E)$ is also regular.
\end{remark}

In order to get a better feeling of the Linear Holonomy Groupoid, we present bellow two extreme examples where the groupoid can be explicitly described.

%%%%% remark %%%%%
\begin{example}[Regular case around closed leaf]\label{regular-case-2}
When $\F$ is regular and $B = L$ is a closed leaf of $\F$  we point out that $\mathcal{G}^\ell$ is in fact the linearization (see \cite{Crainic-Struchiner}) 
of the (usual) holonomy groupoid of the foliation $\F$ around $L$. In other words, $\mathcal{G}^\ell$ is isomorphic to the transformation groupoid 
$\mathrm{Hol}(\F)|_{L} \ltimes  \nu L \rightrightarrows \nu L$ where $\mathrm{Hol}(\F)|_{L} \rightrightarrows L$ is the restriction of the holonomy groupoid to $L$. 

In order to show this, we recall that $E$ is just the normal bundle $\nu(L)$, and connection $\nabla^\tau$ is the Bott connection of the foliation 
(see Remark \ref{ex:regular-case}). Moreover, by counting dimensions we see that the distribution $\mathcal{N}$ must be trivial, or equivalently, the Lie group $K^0_b$ is trivial for all $b \in B$. From this it follows that the lifted foliation $\widetilde{\F}$ on the orthogonal frame bundle $O(E)$ is the horizontal foliation determined by the (flat!) Bott connection. 

With this description we can compute the Lie groupoid $\mathrm{Hol}(\widetilde{\F}) \rightrightarrows O(E)$ explicitly. One obtains that
\[\mathrm{Hol}(\widetilde{\F}) = O(E)\times_L\mathrm{Hol}(\F)|_L = \{(\xi, [\alpha]): \pi(\xi) = \alpha(0)\},\]
where $\pi: O(E) \to L$ denotes the frame bundle projection. The source of $(\xi, [\alpha])$ is $\xi$ and its target is $\widetilde{\alpha}_\xi(1)$ where $\widetilde{\alpha}_\xi$ is the horizontal lift of $\alpha$ to $O(E)$ starting at the frame $\xi$. The product is 
\[(\xi_2, [\alpha_2])\cdot (\xi_1, [\alpha_1]) = (\xi_1, [\alpha_2 \ast \alpha_1]),\]
where $\alpha_2 \ast \alpha_1$ denotes the concatenation of $\alpha_1$ with $\alpha_2$.

It is then clear that $\mathrm{Hol}(\widetilde{\F}) / O(n)$ is isomorphic to $\mathrm{Hol}(\F)|_L$ and that under this identification, both representations on $\nu(L)=E$ coincide.

\end{example}

\begin{example}[Around a Fixed Point]
The example above is an extreme example because of the triviality of the of the group $K^0$. We now explain the other extreme case where the connection is trivial.

Let $\F$ be a SRF foliation on $M$ with a fixed point $x$, i.e., such that $L_x = \{x\}$, and let $B = \{x\}$. In this case, $E = T_xM$ is a vector space with an inner action, and the linearized foliation is given by the orbits of the subgroup $K^0$ of the orthogonal group of the vector space $T_xM$. The frame bundle in this case is just the orthogonal group itself, and the lifted foliation is $\widetilde{F} = \{K^0a: a \in O(T_xM)\}$. It is easy to check that this foliation has trivial holonomy groups, and therefore, its holonomy groupoid is given by
\[\mathrm{Hol}(\widetilde{\F}) = \{(ka,a): a \in O(T_xM), \text{ and } k \in K^0\}.\] 
The source of an arrow $(ka,a)$ is $a$ while its target is $ka$, and the product is given by $(k'ka, ka)\cdot(ka, a) = (k'ka, a)$. In other words, we can identify $\mathrm{Hol}(\widetilde{\F})$ with the action groupoid of the action of $K^0$ on $O(T_xM)$.

It now follows easily that $\mathcal{G} = \mathrm{Hol}(\widetilde{\F})/O(T_xM) = K^0$ (seen as a Lie groupoid over $\{x\}$), and that $\mathcal{G}^\ell = K^0 \ltimes T_xM \rightrightarrows T_xM$ is the action groupoid associated to the representation of $K^0$ on $T_xM$.    
\end{example}

%%%%%%%%%%%%%%% Proof of Theorem %%%%%%%%%%%%%%%
\subsection{Proof of Theorem \ref{foliated-Ambrose-and-Singer}}
\label{Sec-foliated-Ambrose-and-Singer}

As in Theorem \ref{proposition-groupoid-linear-foliation}, let $O(E)$ be the orthogonal frame bundle associated to $E \to B$ and $\mathcal{H}^{\tau}$ the horizontal distribution on $O(E)$ along $T \F_B$ induced by the partial linear connection $\nabla^{\tau}$ on $E$.

Note that a parallel transport along a regular curve $\alpha\subset L$ (i.e., a curve that is a integral line of a vector field tangent to the leaves of $\F_{B}$) can be described by a linearized flow $t \to \varphi_{t}$ that can be lifted to a flow $t \to \varphi_{t}$ on $O(E)$. Let $\widetilde{\F} = \{\widetilde{L}_{\xi_b}\}_{\xi_b \in O(E)}$ be the singular foliation obtained as the orbits of the pseudo-group generated by these lifted flows. Our goal is to prove that $\widetilde{\F}$ is a regular foliation. Once we have proved this we can follow the same argument as in the proof of Theorem \ref{proposition-groupoid-linear-foliation}, i.e., we can define $\widetilde{\mathcal{G}} = \mathrm{Hol}(\widetilde{\F}) \rightrightarrows O(E)$ as the holonomy Lie groupoid of $\widetilde{\F}$, set $\mathcal{G}$ as the Lie groupoid $\widetilde{\mathcal{G}} / O(n) \rightrightarrows B$ and the desired Lie groupoid $\mathcal{G}^{\tau}$ will be $\mathcal{G} \ltimes E$. Note that for each $b \in B$ the holonomy group $\text{Hol}^{\tau}_{b}$ (with respect to $\nabla^{\tau}$) acts effectively and freely  on $O(E)_b$ and the orbits of this action coincide with the intersection of the leaves of $\widetilde{\F}$ with $O(E)_b$. In particular
\begin{equation} \label{eq-1-proof-theorem-holonomy}
\dim T_{\xi_{b}} \big( O(E)_b\cap \widetilde{L}_{\xi_{b}} \big) = \dim \mathfrak{hol}_{b}
\end{equation}
where $\mathfrak{hol}^{\tau}_{b}$ denotes de Lie algebra of the holonomy group $\text{Hol}^{\tau}_{b}$.

Equations \eqref{eq-1-lemma-groupoid-linear-foliation} (that also holds here), \eqref{eq-1-proof-theorem-holonomy}, and the fact that the dimension of $\widetilde{\F}$ is  lower semi-continuous, allow us to conclude that for a slice $S_{b}$ of $\F_{B}$ at $b$ and  each $y\in S_{b}$ close to $b$
\begin{equation}
\label{eq1-holonomy-constant}
\dim \mathfrak{hol}^{\tau}_{y}\geq \dim \mathfrak{hol}^{\tau}_{b}.
\end{equation}

Since $\F_{B}$ is dense on $B$, we can find a sequence $(b_{n})_{n \in \mathbb{N}}$ of points in $S_{b}$ so that $b_{n} \to y$ and $b_{n} \in L_{b}$ from which $b_{n} \in L_{b}$ we conclude that
\begin{equation} \label{eq2-holonomy-constant}
\dim \mathfrak{hol}^{\tau}_{b_{n}}=\dim \mathfrak{hol}^{\tau}_{b}.
\end{equation}
On the other hand, by replacing $b$ with $y$ and $y$ with $b_{n}$ in Equation \eqref{eq1-holonomy-constant}, we have that
\begin{equation} \label{eq3-holonomy-constant}
\dim \mathfrak{hol}^{\tau}_{b_{n}}\geq \dim \mathfrak{hol}^{\tau}_{y}.
\end{equation}
These equations together imply: $$\dim \mathfrak{hol}^{\tau}_{b} = \dim \mathfrak{hol}^{\tau}_{b_{n}} \geq \dim \mathfrak{hol}^{\tau}_{y} \geq \dim \mathfrak{hol}^{\tau}_{b}.$$ Therefore $\dim \mathfrak{hol}^{\tau}_{y}=\dim \mathfrak{hol}^{\tau}_{b}$ for $y\in S_{b}$ near $b$. This fact, together with Equations \eqref{eq-1-lemma-groupoid-linear-foliation} and \eqref{eq-1-proof-theorem-holonomy} imply that the foliation $\widetilde{\F}$ is a regular foliation.

%%%%%%%%%%%%%%%%%%%%%%%%%%%%%%%%%%%%%%%%%%%%%%%%%%%%%%%%
%%%%LIE GROUPOID STRUCTURE OF THE CLOSURE %%%%%%%%%%%%%%
%%%%%%%%%%%%%%%%%%%%%%%%%%%%%%%%%%%%%%%%%%%%%%%%%%%%%%%%
\section{Lie groupoid structure of $\ol{\Flin}$}
\label{section-subgroupoid-closure}
In this section we show that if $B=\ol{L}$ is the closure of a leaf $L$ in $M$, then the leaf closure foliation $\ol{\Flin}$ comes from a proper Lie groupoid. 
Moreover, the linear holonomy groupoid $\mathcal{G}^{\ell}$ constructed in Section \ref{section-lie-groupoid-structure} is a subgroupoid of the groupoid $\ol{\mathcal{G}^{\ell}}$ describing $\ol{\Flin}$, and in fact, $\mathcal{G}^{\ell}$ is dense in $\ol{\mathcal{G}^{\ell}}$.

The proof of this result will rely on a similar statement for regular Riemannian foliations which can be seen as an extension of Molino's Theorem about leaf closures \cite{Molino}. We include the proof of this particular case in Section \ref{closure-regular-case} and proceed to the singular case in Section \ref{around-leaf-closure}.

\subsection{Closure on the regular case}\label{subsec-regular-case}
The structure theory for regular Riemannian foliations developed by Molino (\cite{Molino}) has as a consequence that the foliation $\ol{\F}$ given by the leaf closures of a regular Riemannian foliation $\F$ is itself a (possibly singular) Riemannian foliation. However, a bit more is true: the leaves of are the orbits of a proper Lie groupoid. This fact seems to be well known to the community working in the intersection of Lie groupoid theory and Riemannian foliations (see for example \cite{Wang}). Here we present a proof of this result in the spirit of the previous constructions of this paper.
 
\begin{proposition}\label{closure-regular-case}
 Let $\F$ be a regular Riemannian foliation on a complete manifold $(M,g)$.
 Then the leaves of the singular foliation $\ol{\F}$ are orbits of a proper Lie groupoid $\ol{\mathrm{Hol}({\F})}\rightrightarrows M$.
 Moreover, $\mathrm{Hol}(\F)$ is a dense subgroupoid of $\ol{\mathrm{Hol}(\F)}$. 
 \end{proposition}

\begin{proof}
Let $\F$ be a (regular) Riemannian foliation and denote by $E$ the normal bundle of $\F$, and by $\widetilde{\F}$ the lifted foliation on $O(E)$ using the Bott connection (see Example \ref{ex:regular-case}). This foliation is transversally parallelizable and this implies that $\ol{\widetilde{\F}}$ is a simple regular foliation \cite{Molino}.

Since $\ol{\widetilde{\F}}$ is a simple foliation and the restriction of $\widetilde{\F}$ to a leaf closure is a Lie foliation, then we conclude that both $\ol{\widetilde{\F}}$ and $\widetilde{\F}$ have trivial leafwise holonomy.
In other words, for each pair of points in a leaf there is a single holonomy class of a path connecting them and therefore the holonomy groupoid of $\ol{\widetilde{\F}}$ is just the set of pairs of points on the same leaf, and therefore is a proper Lie groupoid. 
Moreover, the homomorphism $\mathrm{Hol}(\widetilde{\F})\to\mathrm{Hol}(\ol{\widetilde{\F}})$ is injective.

We now show that $\mathrm{Hol}(\widetilde{\F})$ is dense in $\mathrm{Hol}(\ol{\widetilde{F}})$.
Let $x\xrightarrow{g} y$ be an arrow of $\mathrm{Hol}(\ol{\widetilde{\F}})$, this means that $x,y$ are in a same leaf closure $\ol{\widetilde{L}}$. 
There exists sequences $x_n, y_n$ in $\widetilde{L}$ converging to $x$ and $y$ respectively, and a sequence of path holonomies  $x_n\xrightarrow{h_n} y_n$.
Using the fact that $\mathrm{Hol}(\ol{\widetilde{\F}})$ is proper and has trivial isotropies we conclude that $h_n$ converges to $g$.

The action of $O(n)$ on $O(E)$ preserves the foliation $\widetilde{\F}$, and consequently preserves $\ol{\widetilde{\F}}$. 
It extended to a free and proper action  on $\mathrm{Hol}(\ol{\widetilde{\F}})\rightrightarrows O(E)$ by automorphisms. 
It follows that we obtain a commutative diagram of Lie groupoids 
\[\xymatrix{
\mathrm{Hol}(\ol{\widetilde{\F}}) \ar@<0.25pc>[r] \ar@<-0.25pc>[r] \ar[d] & O(E)  \ar[d]   \\
\mathrm{Hol}(\ol{\widetilde{\F }})/O(q) \ar@<0.25pc>[r] \ar@<-0.25pc>[r] & M .
}\]

%%%%%%%%% 
 Let $\ol{\mathrm{Hol}(\F)}$ be the Lie groupoid $\mathrm{Hol}(\ol{\widetilde{\F }})/O(q) \rightrightarrows M.$ 
%The orbits of the Lie groupoid $\ol{\mathrm{Hol}(\F)}:=\mathrm{Hol}(\ol{\widetilde{\F}})/O(n)\rightrightarrows M$ are of the projections of the leaf closures.
Since the leaves of the lifted foliation $\tilde{\F}$ projects into leaves of $\F$, and the projection $O(E) \xrightarrow{\pi} M$ is closed follows that 
$\ol{L} = \pi(\ol{\widetilde{L}})$ showing that the orbits of $\ol{\mathrm{Hol}(\F)}$ are the leaf closures of $\F$. Moreover, since $O(n)$ is a compact Lie group it follows also that $\ol{\mathrm{Hol}(\F)}$ is proper.

We deduce that $\mathrm{Hol}(\F)$ is a dense subgroupoid of $\ol{\mathrm{Hol}(\F)}$ from the fact that  $\mathrm{Hol}(\widetilde{\F})$ is a subgroupoid of $\mathrm{Hol}(\ol{\widetilde{\F}})$.

\end{proof}

\begin{remark}\label{rmk:foliation-with-symmetries}
If we assume moreover that  $M\curvearrowleft K$ is a right action by isometries which preserves $\F$, then $K$ acts on the right of $\ol{\mathrm{Hol}(\F)}\rightrightarrows M$ by automorphims. 
In fact, if $\F$ is invariant under the isometric action of $K$ on $M$, then for each $k$ in $K$ we have an isometry 
$E_x \xrightarrow{d R_{k}} E_{x\cdot k}$.
This defines a right action on $O(E)\curvearrowleft K$ which assigns a frame $\mathbb{R}^{n}\xrightarrow{\xi_x} E_x$ to 
\[
\xymatrix{ \mathbb{R}^{n}\ar[r]^{\xi_x}\ar@/^1.2pc/[rr]^{d R_{k}\circ \xi_x} &  E_x \ar[r]^{d R_{k}} & E_{x\cdot k}. }
\]
Since $K$ preserves $\F$ the above action preserves $\widetilde\F$.
Consequently this action sends closures to closures, and therefore $K$ acts by automorphims on $\mathrm{Hol}(\ol{\F'})$.
Note also that for $k$ in $K$ the map $k:O(E)\to O(E)$ is a $O(n)$-equivariant map.
So, for $\xi$ in $O(E)$ and $A$ in $O(n)$, we have 

$$
(\xi A)\cdot k= d R_{k}\circ (\xi \circ A) =  (d R_{k}\circ \xi)\circ A = (\xi \cdot k)A.
$$
Therefore, these actions commute, and the action of $K$ on $\mathrm{Hol}(\ol{\widetilde{\F}})$ induces a $K$ action on $\ol{\mathrm{Hol}(\F)}$ by automorphims.
\end{remark}

\subsection{Around the closure of a leaf}\label{around-leaf-closure}

We now use the results of Section \ref{subsec-regular-case} to generalize Proposition \ref{closure-regular-case} to the linearization of a SRF around the closure of a leaf.

\begin{theorem} % \label{prop:subgrupoids}
Let $\F$ be a singular Riemannian foliation on a complete manifold $(M,g)$ and $B = \ol{L}$. 
Then there exists a proper Lie groupoid $\overline{\mathcal{G}^{\ell}}$ over  a saturated $\epsilon$-tubular neighborhood $U$ of $B$ whose orbits are 
 the leaves  of $\overline{\F^{\ell}}$. In addition $\mathcal{G}^{\ell}$ is a dense Lie subgroupoid of $\overline{\mathcal{G}^{\ell}}$.  

\end{theorem}

\begin{proof}
Since the foliation $\F_B$ is dense we can us Lemma \ref{lemma-total-connection} to produce a Riemannian metric on $O(E)$ such that the lifted foliation $\widetilde{F}$ is a regular Riemannian foliation.

Since $\widetilde{\F}$ is a Riemannian foliation, Proposition \ref{closure-regular-case} implies that there exists a proper Lie groupoid $\ol{\mathrm{Hol}(\widetilde{\F})}$ such that its orbit foliation is $\ol{\widetilde{\F}}$ and such that  $\mathrm{Hol}(\widetilde{\F})$ is a dense  subgroupoid. 
% Following Remark \ref{foliation-with-symmetries} and the proof of the Proposition \ref{proposition-groupoid-linear-foliation} we can take the quotient by $O(n)$ and conclude that $\mathrm{Hol}(\widetilde{\F})/O(q)$ is a subgroupoid of $\ol{\mathrm{Hol}(\widetilde{\F})}/O(n)$.

By taking the product of both of these Lie groupoids with the trivial Lie groupoid $\mathbb{R}^n\rightrightarrows \mathbb{R}^n$ we obtain an injective morphism of Lie groupoids
\[\xymatrix{
\mathrm{Hol}(\widetilde{\F}) \times \mathbb{R}^n \ar@<0.25pc>[dr] \ar@<-0.25pc>[dr] \ar[rr] & & \ol{\mathrm{Hol}(\widetilde{\F})} \times \mathbb{R}^n \ar@<0.25pc>[dl] \ar@<-0.25pc>[dl] \\
& O(E)\times \mathbb{R}^n .&}
\]

From Remark \ref{rmk:foliation-with-symmetries} and Proposition \ref{prop:quotient} we can take the quotient by $O(n)$ to obtain 
\[\xymatrix{
\mathrm{Hol}(\widetilde{\F})\times{\mathbb{R}^n} \ar[r] \ar[d] & \ol{\mathrm{Hol}(\widetilde{\F})} \times \mathbb{R}^n \ar[d]^{\pi} \\
\mathrm{Hol}(\widetilde{\F})/O(n)\ltimes E \ar[r] & \ol{\mathrm{Hol}(\widetilde{\F})}/O(n)\ltimes E.
}\]
This shows that $\mathcal{G}^{\ell}=\mathrm{Hol}(\widetilde{\F})/O(n)\ltimes E $ is a subgroupoid of $\ol{\mathcal{G}^{\ell}}:=\ol{\mathrm{Hol}(\widetilde{\F})}/O(n)\ltimes E$. Moreover, since $\ol{\mathrm{Hol}(\widetilde{\F})}$ is proper and the transformation groupoid of an action of a proper groupoid is again proper, it follows that $\ol{\mathcal{G}^{\ell}}$ is a proper Lie groupoid.

We will now show that $\mathcal{G}^{\ell}$ is dense in $\ol{\mathcal{G}^{\ell}}$. Let $g$ be an arrow of $\ol{\mathcal{G}^{\ell}}$. 
Fix an arrow $(\widetilde{g},v)$ in $\ol{\mathrm{Hol}(\widetilde{\F})}\times{\mathbb{R}^n}$ with $\pi(\widetilde{g},v)=g$.
By density there exists a sequence $\widetilde{h}_n$ in $\mathrm{Hol}(\widetilde{\F})$ converging to $\widetilde{g}$.
Setting $h_n=\pi(\widetilde{h}_n,v)$ we get a sequence on $\mathcal{G}^{\ell}$ converging to $g$, showing that $\mathcal{G}^{\ell}$ is dense in $\ol{\mathcal{G}^{\ell}}$.

Finally, in order to show that the orbits of $\ol{\mathcal{G}^{\ell}}$ are the leaf closures of the orbits of $\mathcal{G}^{\ell}$, we identify $E$ with the associated bundle $O(E) \times_{O(n)} \mathbb{R}^n$ and note that the natural projection map $\pi: O(E) \times \mathbb{R}^n \to E$ is closed. Since each leaf $L$ of $\Flin$ can be seen as $L=\pi(\tilde{L}\times\{v\})$ for a leaf $\tilde{L}$ of $\tilde{\F}$, it follows that $\ol{L} = \pi(\ol{\widetilde{L}}\times \{v\})$, and hence the orbit foliation of $\ol{\mathcal{G}^{\ell}}$ is $\ol{\Flin}$.

\end{proof}

%%%%%%%%%%%%%%%%%%%%%%%%%%%%%%%%%%%%%%%%%%%%%%%%%%%%%%%%%%%%%%%%%%%%%%%%%%%%%%
%%%%%%%%%%%%%%% THE STRUCTURE OF THE LINEAR HOLONOMY GROUPOID %%%%%%%%%%%%%%%
%%%%%%%%%%%%%%%%%%%%%%%%%%%%%%%%%%%%%%%%%%%%%%%%%%%%%%%%%%%%%%%%%%%%%%%%%%%%%%
\section{The Lie Algebroid Associated to The Infinitesimal Data} \label{section-rotate-translate-groupoid}

In this section, we discuss the Lie algebroid of the linear holonomy groupoid of a SRF, see Definition \ref{definition-algebroid}. The construction holds in a slightly more general setting where we drop any metric condition on the data. In what we present bellow, we will avoid a direct use of the SRF by using only the infinitesimal data that one obtains from the semi-local model of a SRF.  The main ingredients of the construction are as follows:

\begin{enumerate}
\item[(a)] A rank $n$ vector bundle  $E \to B$;
\item[(b)] a (regular) foliation $\F_B$ on $B$;
\item[(c)] a $\F_B$-partial linear connection $\nabla^{\tau}: \mathfrak{X}(\F_{B})\times \Gamma(E)\to \Gamma(E)$;
\item[(d)] a bundle of Lie algebras $\mathfrak{k} \to B$ such that:
\begin{enumerate}
\item[(d1)] $\mathfrak{hol}^{\tau}_{b} \subset \mathfrak{k}_{b}\subset \mathrm{End} (E_b)$, $\forall b \in B$, where $\mathfrak{hol}^{\tau}_{b}$ is the (leafwise) $\nabla$-holonomy Lie algebra of $\nabla^{\tau}$, $\mathfrak{k}_b$ is the fiber of $\mathfrak{k}$ over $b \in B$, and $\mathrm{End}(E_b)$ denotes the Lie algebra of endomorphisms of $E_b$, 
%\textbf{with Lie brackets defined to induced
%a morphism between $\mathrm{End}(E_b)$ and the  Lie algebra of fundamental vector
%fields  induced by its action};
\item[(d2)] For all $X \in \mathfrak{X}(\F_B)$, $\Gamma(\mathfrak{k})$ is invariant by elements of the form $\nabla^\tau_X$ under the commutator bracket of $\Gamma(\mathrm{End}(E))$, i.e., 
\[\nabla_X^\tau \phi = \nabla_X^\tau \circ \phi - \phi \circ \nabla_X^\tau \in \Gamma(\mathfrak{k})\]
for all $X \in \mathfrak{X}(\F_B)$ and $\phi \in \Gamma(\mathfrak{k})$.

\end{enumerate}

\end{enumerate}

%%%explain what we want to do
Using these ingredients we build an integrable Lie algebroid $A(\nabla^{\tau},\mathfrak{k}) \to B$ which satisfies the following properties:
\begin{itemize}
\item $A(\nabla^{\tau},\mathfrak{k})$ fits into an exact sequence of Lie algebroids
\[0 \longrightarrow \mathfrak{k} \longrightarrow A(\nabla^{\tau},\mathfrak{k}) \longrightarrow T\F_B \longrightarrow 0;\]
\item $A(\nabla^{\tau},\mathfrak{k})$ is a Lie subalgebroid of the Atiyah algebroid $\mathfrak{gl}(E)$ of the frame bundle of $E$.
\end{itemize}

\begin{remark}
If we start with a SRF $\F$ on $M$, and a closed saturated submanifold $B$ contained in a stratum of $M$, then we obtain the infinitesimal data from the semi-local model of $\F$ around $B$. In this case, $E$ is the normal bundle to $B$ in $M$, $\nabla^\tau$ is the $\F_B$-partial connection defined in the Proposition \ref{proposition-connections-tau-affine}, and $\mathfrak{k}_b$ is the Lie algebra of the Lie group $K^0_b$. For this particular example it follows that $A(\nabla^{\tau},\mathfrak{k})$ will be the Lie algebroid of the Lie groupoid $\mathcal{G} = \mathrm{Hol}(\widetilde{F})/O(n) \rightrightarrows B$, and the inclusion of $A(\nabla^{\tau},\mathfrak{k})$ is the restriction of the differential of the representation map $\mathcal{G} \to \mathrm{GL}(E)$. It then follows that the Lie algebroid of the linear holonomy groupoid $\mathcal{G}^\ell$ is the action algebroid associated to the representation $A(\nabla^{\tau},\mathfrak{k}) \to \mathfrak{gl}(E)$.
\end{remark}

\begin{remark}
%%%%
%%%%
Before we present the formal construction of our Lie algebroid, let us briefly give an intuition of how it appear 
in the case of $\mathcal{G} = \mathrm{Hol}(\widetilde{\F})/O(n) \rightrightarrows B$. 
As explained in Section \ref{section-lie-groupoid-structure}  
each leaf of the foliation $\widetilde{\F}$ on $O(E)$ is invariant under the action of $K^{0}$
and the basic distribution $\mathcal{H}^{\tau}$ is tangent to it. These facts  allow us to 
identify $T\widetilde{\mathcal{F}} $ with $\mathcal{H}^{\tau}\oplus\mathfrak{k}$. With this identification, we can compute the Lie bracket of $O(n)$-invariant vector fields tangent to $\widetilde{\F}$ in terms of their components. Passing to the quotient, we obtain a Lie bracket on the vector bundle 
$$T\mathcal{F}_{B}\oplus \mathfrak{k}=(\mathcal{H}^{\tau}\oplus\mathfrak{k})/O(n)=T\widetilde{\mathcal{F}}/O(n)\to O(E)/O(n)=B.$$ 

\end{remark}

\begin{remark}
Our construction can be thought of as a foliated version of the Ambrose-Singer reduction theorem. In fact, one can restate the classical theorem as follows: if $\nabla$ is a linear connection on a rank $n$ vector bundle $E \to B$, then for any Lie algebra $\mathfrak{k}$ such that $\mathfrak{hol}^\nabla \subset \mathfrak{k} \subset \mathfrak{gl}(n)$, there exists transitive Lie subalgebroid $A(\nabla, \mathfrak{k})$ of the Atiyah algebroid $\mathfrak{gl}(E)$ such that the isotropy Lie algebras of  $A(\nabla, \mathfrak{k})$ are all isomorphic to $\mathfrak{k}$. In the case of a foliated connection me must replace the Lie algebra above by a bundle of Lie algebras $\mathfrak{k}$ which contains the possibly singular bundle of Lie algebras $\mathfrak{hol}^\tau$.
\end{remark}

%%
%%% Foliated Atiyah algebroid
%%
The first step needed in the construction of $A(\nabla^{\tau},\mathfrak{k})$ is a foliated version of the Atiyah algebroid $\mathfrak{gl}(E)$. For a vector bundle $E$ over a foliated manifold $B$ we define the \emph{foliated general linear algebroid} $\mathfrak{gl}(E,\F_B) \to B$ as follows. As a vector bundle,  $\mathfrak{gl}(E,\F_B)$ is the fibered product with respect to the anchor $\rho$ of $T\F_B$ with $\mathfrak{gl}(E)$, i.e,
\[\mathfrak{gl}(E,\F_B) = \{(X, D) \in T\F_B \oplus \mathfrak{gl}(E): \rho(D) = X\}.\]
We remark that $\mathfrak{gl}(E, \F_B)$ sits in a short exact sequence
\begin{equation}\label{Atiyah-sequence}
0 \longrightarrow \mathrm{End}(E) \longrightarrow \mathfrak{gl}(E, \F_B) \longrightarrow T\F_B \longrightarrow 0,
\end{equation}
and therefore $\mathfrak{gl}(E, \F_B)$ is a (smooth) vector bundle. The space of sections of $\mathfrak{gl}(E, \F_B)$ identifies with the space of \emph{$\F_B$-compatible derivations of $E$}: $\mathbb{R}$-linear operators $D: \Gamma(E) \to \Gamma(E)$ such that there exists a $\F_B$ foliated vector field $X \in \mathfrak{X}(\F_B)$ for which 
\[D(fs)=fD(s)+X_D(f)s\] 
for all $f$ in $C^{\infty}(B)$ and $s$ in $\Gamma(E)$. The Lie bracket of $\Gamma(\mathfrak{gl}(E, \F_B))$ is the commutator bracket of derivations. Finally, the anchor of $\mathfrak{gl}(E, \F_B)$ is $\rho(X,D) = X$.

The main purpose of considering the foliated general linear algebroid is that there is a  one-to-one correspondence between splittings of the foliated Atiyah sequence \eqref{Atiyah-sequence} and $\F_B$-partial connections on $E$ given by
\[\nabla^\tau_Xs = \tau(X)(s)\]
for any splitting $\tau: T\F_B \to \mathfrak{gl}(E, \F_B)$.

It follows that a choice of a $\F_B$-partial connection induces an identification of vector bundles $\mathfrak{gl}(E, \F_B) \simeq T\F_B \oplus \mathrm{End}(E)$. Under this identification we can re-express the anchor $\mathfrak{gl}(E, \F_B)$ as $\rho(X, \phi) = X$. The Lie bracket on the space of sections can also be re-expressed as
\[[(X_1,\phi_1), (X_2, \phi_2)] = ([X_1,X_2]_{\mathfrak{X}(\F_B)}, [\phi_1,\phi_2]_{\mathrm{End}(E)} + \nabla^\tau_{X_1} \phi_2 - \nabla^\tau_{X_2} \phi_1 - R^{\nabla^\tau}(X_1,X_2)),\]
for all $(X_1,\phi_1), (X_2,\phi_2) \in \Gamma(T\F_B \oplus \mathrm{End}(E))$, where 
\[R^{\nabla^\tau} (X_1,X_2)= \nabla^\tau_{[X_1,X_2]} - [\nabla^\tau_{X_1},\nabla^\tau_{X_2}] \in \Gamma(\mathrm{End}(E))\]
denotes the curvature of $\nabla^\tau$, and 
\[\nabla_X^\tau \phi = \nabla_X^\tau \circ \phi - \phi \circ \nabla_X^\tau \in \Gamma(\mathrm{End}(E))\]
is the induced $\F_B$-partial linear connection on $\mathrm{End}(E)$.

It is by now clear how to construct the Lie subalgebroid $A(\nabla^{\tau},\mathfrak{k})$ of $\mathfrak{gl}(E)$. 

\begin{definition}
\label{definition-algebroid}
As a vector bundle we take $A(\nabla^{\tau},\mathfrak{k}) = T\F_B \oplus \mathfrak{k}$. Its anchor map is the projection onto the first factor, and its bracket is given by the restriction of the bracket on $\mathfrak{gl}(E, \F_B)$, i.e.,
\[[(X_1,\phi_1), (X_2, \phi_2)] = ([X_1,X_2]_{\mathfrak{X}(\F_B)}, [\phi_1,\phi_2]_{\mathfrak{k}} + \nabla^\tau_{X_1} \phi_2 - \nabla^\tau_{X_2} \phi_1 - R^{\nabla^\tau}(X_1,X_2)),\]
for all $(X_1,\phi_1), (X_2,\phi_2) \in \Gamma(T\F_B \oplus \mathfrak{k})$.
\end{definition}

\begin{proposition}
\label{proposition-algebroid-puro}
$A(\nabla^{\tau},\mathfrak{k})$ is a Lie subalgebroid of $\mathfrak{gl}(E)$.
\end{proposition}

\begin{proof}
Since $\mathfrak{gl}(E,\F_B)$ is a Lie subalgebroid of $\mathfrak{gl}(E)$, it suffices to show that $A(\nabla^{\tau},\mathfrak{k})$ is a Lie subalgebroid of $\mathfrak{gl}(E,\F_B)$. Therefore we must show that $\Gamma(A(\nabla^{\tau},\mathfrak{k}))$ is closed under the Lie bracket of $\Gamma(\mathfrak{gl}(E,\F_B))$. This boils down to 
\begin{itemize}
\item $\mathfrak{k}$ is a sub bundle of Lie algebras of $\mathrm{End}(E)$	: this is part of condition (d1);
\item $R^{\nabla^\tau}(X_1, X_2)$ takes value in $\mathfrak{k}$ for all $X_1, X_2 \in \mathfrak{X}(\F_B)$: this follows from the fact that $\mathfrak{hol}^\tau_b \subset \mathfrak{k}_b$ for all $b \in B$, which is also part of condition (d1);
\item $\nabla^\tau_X\phi$ belongs to $\Gamma(\mathfrak{k})$ for all $X\in \mathfrak{X}(\F_B)$, and $\phi \in \Gamma(\mathfrak{k})$: this is condition (d2).
\end{itemize}
\end{proof}

\begin{remark}
In the case where $E$ is an Euclidean vector bundle, it is common to consider the Lie subalgebroid $\mathfrak{o}(E)$ instead of $\mathfrak{gl}(E)$, where $\mathfrak{o}(E)$ is the Lie subalgebroid whose space of sections is the subspace of derivations of $E$ which satisfy
\[X_D\langle s_1, s_2 \rangle = \langle D(s_1), s_2 \rangle + \langle s_1, D(s_2) \rangle \text{ for all } s_1, s_2 \in \Gamma(E).\]
If we use the fiberwise metric on $E$ to identify $\mathrm{End}(E)$ with $E^* \otimes E^*$, then $\mathfrak{o}(E)$ is the Lie subalgebroid which fits into the exact sequence
\[0 \longrightarrow E^*\wedge E^* \longrightarrow \mathfrak{o}(E) \longrightarrow TB \longrightarrow 0,\]
and splittings of this sequence are in one-to-one correspondence with linear connections which are compatible with the fiberwise metric on $E$. The Lie subalgebroid $\mathfrak{o}(E)$ is the Lie algebroid of the Lie groupoid $\mathcal{O}(E) \rightrightarrows B$ whose arrows consist of linear isometries between the fibers of $E$ (see Example \ref{ex:gl-groupoid}).

When $B$ is a foliated manifold we may construct a Lie subalgebroid $\mathfrak{o}(E, \F_B)$ which is analogous to $\mathfrak{gl}(E, \F_B)$. Splittings of the corresponding short exact sequence are in one-to-one correspondence with $\F_B$-partial connections compatible with the fiberwise metric on $E$.

Finally, we remark that the infinitesimal data associated to a SRF $\F$ around a closed saturated submanifold $B$ of a regular stratum satisfies a stronger version of conditions (a) through (d) of the beginning of this section. In this case $\nabla^\tau$ is compatible with the metric on $E$, and $\mathfrak{k}$ is a sub bundle of Lie algebras of $E^*\otimes E^*$. It then follows that $A(\nabla^\tau, \mathfrak{k})$ is a Lie subalgebroid of $\mathfrak{o}(E)$. 
\end{remark}

\begin{proposition}
\label{proposition-algebroid-SRF}
When the infinitesimal data (a)-(d) come from a SRF $\F$ around a closed saturated submanifold $B$ of a regular stratum, then $A(\nabla^\tau, \mathfrak{k})$ is the Lie algebroid of the Lie groupoid $\mathcal{G} = \mathrm{Hol}(\widetilde{F})/O(n) \rightrightarrows B$ constructed in Section \ref{section-lie-groupoid-structure}.
\end{proposition}

\begin{proof}
We consider $A(\nabla^\tau, \mathfrak{k})$ as a Lie subalgebroid $\mathfrak{o}(E)$ and follow the integration scheme for Lie subalgebroids developed in \cite{Moerdijk-Mrcun-subalgebroid} and described at the end of Section \ref{section-facts-algebroids} of this paper.

The source fibers of $\mathcal{O}(E)$ identify with the orthogonal frame bundle of $O(E)$, and under this identification the right invariant foliation $\F^{A(\nabla^\tau, \mathfrak{k})}$ on $\mathcal{O}(E)$ is mapped to to the lifted $\widetilde{F}$ on $O(E)$. It then follows that the Lie groupoid integrating $A(\nabla^\tau, \mathfrak{k})$ obtained by taking $\mathrm{Hol}(\F^{A(\nabla^\tau, \mathfrak{k})})/ \mathcal{O}(E)$ is isomorphic to $\mathcal{G} =  \mathrm{Hol}(\widetilde{F})/O(n)$.
\end{proof}

\begin{remark}
With the explicit description developed here it is now a simple computation the inclusion of $A(\nabla^\tau, \mathfrak{k})$ in $\mathfrak{gl}(E)$ is the restriction of the differential of the Lie groupoid morphism $\mathcal{G} = \mathrm{Hol}(\widetilde{F})/ O(n) \to GL(E)$ induced from the representation of of $\mathcal{G}$ on $E$. It then follows that the Lie algebroid of the linear holonomy groupoid $\mathcal{G}^\ell \rightrightarrows E$ is the action algebroid $A(\nabla^\tau, \mathfrak{k})\ltimes E \to E$. 
\end{remark}

%%%%%%%%%%%%%%%%%%%%%%%%%%%%%%%%%%%%%%%%%%%%
%%%%%%%%%%%%%%% BIBLIOGRAPHY %%%%%%%%%%%%%%%
%%%%%%%%%%%%%%%%%%%%%%%%%%%%%%%%%%%%%%%%%%%%

\bibliographystyle{amsplain}

%%%%%%%%%%%%%%%%%%% THE END %%%%%%%%%%%%%%%%%%%%%
\end{document}